\documentclass[a4paper,12pt]{amsart}
\usepackage[margin=2cm]{geometry}
\usepackage{amsmath}
\usepackage{a4wide}
\usepackage[utf8]{inputenc}
\usepackage{amssymb}
\usepackage{amsopn}
\usepackage{epsfig}
\usepackage{amsfonts}
\usepackage{latexsym}
\usepackage{amsthm}
\usepackage{enumerate}
\usepackage[UKenglish]{babel}
\usepackage{verbatim}
\usepackage{color}
\usepackage{calrsfs}
\usepackage{pgf}
\usepackage{tikz}
\usepackage{breqn}
\usepackage{mathtools}
\usepackage{tikz}
\usepackage{hyperref}
\hypersetup{
	colorlinks=true,  
	linkcolor=blue,   
	citecolor=blue,   
	urlcolor=blue     
}
\usetikzlibrary{arrows,automata}
\DeclareMathAlphabet{\pazocal}{OMS}{zplm}{m}{n}

\newtheorem{theorem}{Theorem}[section]
\newtheorem{lemma}[theorem]{Lemma}
\newtheorem{proposition}[theorem]{Proposition}

\theoremstyle{definition}
\newtheorem{definition}[theorem]{Definition}

\newtheorem{main}{Theorem}

\theoremstyle{remark}
\newtheorem{remark}[theorem]{Remark}

\numberwithin{equation}{section}

\newcommand{\set}[1]{\left\{#1\right\}} 
\newcommand{\qtq}[1]{\quad \text{#1}\quad }

\newcommand{\bj}[1]{\left(#1\right)}

\makeatletter
\def\@tocline#1#2#3#4#5#6#7{\relax
	\ifnum #1>\c@tocdepth
	\else
	\par \addpenalty\@secpenalty\addvspace{#2}
	\begingroup \hyphenpenalty\@M
	\@ifempty{#4}{
		\@tempdima\csname r@tocindent\number#1\endcsname\relax
	}{
		\@tempdima#4\relax
	}
	\parindent\z@ \leftskip#3\relax \advance\leftskip\@tempdima\relax
	\rightskip\@pnumwidth plus4em \parfillskip-\@pnumwidth
	#5\leavevmode\hskip-\@tempdima
	\ifcase #1
	\or\or \hskip 1em \or \hskip 2em \else \hskip 3em \fi
	#6\nobreak\relax
	\dotfill\hbox to\@pnumwidth{\@tocpagenum{#7}}\par
	\nobreak
	\endgroup
	\fi}
\makeatother

\title{ }
\vspace{-5em}
\author[Y. Hu]{YueCai Hu}
\address{\footnotesize{School of Mathematical Sciences, Shenzhen University, Shenzhen 518060, People's Republic of China.}}
\email{2250201005@email.szu.edu.cn}

\author[R. Alcaraz Barrera]{Rafael Alcaraz Barrera}
\address{\footnotesize{Instituto de F\'isica, Universidad Aut\'onoma de San Luis Potos\'i. Av. Parque Chapultepec 1570, Privadas del Pedregal, 78295 San Luis Potos\'i, S.L.P. M\'exico.}}
\email{ralcaraz@ifisica.uaslp.mx}

\author[Y. Zou]{Yuru Zou$^*$}
\address{School of Mathematical Sciences,
	Shenzhen University,
	Shenzhen 518060,
	People's Republic of China}
\email{yuruzou@szu.edu.cn}
	\thanks{$^*$Corresponding Author}
\begin{document}
	\begin{abstract}
		Given two real numbers $q_0,q_1$ with $q_0, q_1 > 1$ satisfying $q_0+q_1 \geq q_0q_1$, we call a sequence $(d_i)$ with $d_i\in\set{0,1}$ a \emph{$(q_0,q_1)$-expansion} or a \emph{double-base expansion} of a real number $x$ if
		\[
		x=\mathop{\sum}\limits_{i=1}^{\infty} \frac{d_{i}}{q_{d_1}q_{d_2}\cdots q_{d_i}}.
		\]
		When $q_0=q_1=q$, the set of \emph{univoque bases} is given by the set of $q$'s such that $x = 1$ has exactly one $(q, q)$-expansion. The topological, dimensional and symbolic properties of such sets and their corresponding sequences have been intensively investigated. 
		In our research, we study the topological and  dimensional properties of the set of univoque bases for double-base expansions.
		This problem is more complicated, requiring new research strategies. Several
		new properties are uncovered. In particular, we show that the set of univoque bases in the double base setting is a meagre set with   full Hausdorff dimension.
	\end{abstract}
	\maketitle
\section{Introduction}
\label{sec:intro}
Since their introduction in the seminal papers by R\'{e}nyi \cite{Ren1957} and Parry \cite{Par1960}, the theory of \emph{expansions in non-integer bases}, colloquially known as \emph{$\beta$-expansions} or \emph{$q$-expansions}, has received much attention from researchers in several areas of mathematics. In particular, $q$-expansions have been studied extensively during the last thirty years within number theory, diophantine approximation, combinatorics on words, ergodic theory, fractal geometry and symbolic dynamics.

Let us consider the mentioned setting for the case when $q \in (1, 2]$. Set
\[
\Sigma_{2}:= \set{0,1}^{\mathbb{N}} = \set{(d_i): d_i \in \set{0,1} \qtq{for each} i \in \mathbb{N}}.
\]
In this paper, $\mathbb{N}$ denotes the set of natural numbers and $\mathbb{N}_0$ the set of non-negative integers.
Given $q\in(1,2]$, we call a sequence $(d_i) \in \Sigma_{2}$ a \emph{$q$-expansion of $x$} if
\begin{equation}
	\label{eq:q-expansion}
	x=\pi_{q}((d_i)):=\mathop{\sum}\limits_{i=1}^\infty\dfrac{d_i}{q^i}.
\end{equation}
It is well known  that for $q\in(1,2]$, a number $x$ has an expansion in base $q$ if and only if $x\in [0, 1/(q-1)]$.  
Probably, the most remarkable attribute of $q$-expansions is the fact that \color{black} they are almost never unique. \color{black} 
In \cite{ErdJooKom1990, Sid2009} it was shown that for any $k\in \mathbb{N} \cup\{\aleph_{0}\}\cup\{2^{\aleph_{0}}\}$ there always exists a base $q\in(1, 2)$ and $x\in [0, 1/(q-1)]$ such that $x$ has precisely $k$ different $q$-expansions. Also, in \cite{Sid2003,DajdeV2007} the generic behavior of the set of $q$-expansions is described: namely, for any $q\in(1,2)$, almost every $x\in[0, 1/(q-1)]$ with respect to the Lebesgue measure has uncountably many $q$-expansions. The situation described previously is substantially different in comparison to the case $q=2$, where every number within the unit interval, except for the countable set of exceptions given by the dyadic rationals, has a unique $2$-expansion.

The study of expansions in non-integer bases gained new impetus after the research performed by Erd\H{o}s, Horv{\'a}th, Jo{\'o}, Komornik and Loreti in \cite{ErdJooKom1990,ErdHorJoo1991,ErdJoo1992,KomLor1998}, where the set of bases $q \in (1,2]$ for which $x=1$ has a unique $q$-expansion was studied. The base $q$ is called a \emph{univoque base}. Since then, the 
topological and combinatorial
properties of the set of univoque bases $\mathcal{U}$  have received significant attention; see \cite{ErdJooKom1990,ErdHorJoo1991,ErdJoo1992,GleSid2001,KomLor2007,DeVKom2009,KomKonLi2015,AlcBakKon2016,DeVKomLor2016,All2017,AlcBar2021,DeVKomLor2022} and references therein.
We will sum up some significant properties of the set $\mathcal{U}$ below. To do so, we need  to recall some notations. 

By a word $\omega$ we mean a finite string of digits $\omega_1\cdots\omega_n$ with each $\omega_i\in \{0,1\}$.  The length of $\omega$  is denoted  by $|\omega|$. The \emph{empty word} will be denoted by $\epsilon$. We define the \emph{concatenation} of two finite words $\upsilon = u_1 \ldots u_m$ and $\nu = v_1 \ldots v_n$ as the word $\upsilon \nu = u_1 \ldots u_m v_1 \ldots v_n$. Also, for a finite word $\omega$ and $n \in \mathbb{N}$, $\omega^n$ stands for the word obtained by concatenating $\omega$ with itself $n$ times, and $\omega^\infty$ stands for the sequence obtained by concatenating $\omega$ with itself infinitely many times. \color{black} Let $\omega = \omega_1 \cdots \omega_n$ be a word where $\omega_n = 0$. We define $\omega^+ = \omega_1 \cdots \omega_n^+ = \omega_1 \cdots (\omega_n + 1)$. Similarly, if $\omega_n = 1$, we define $\omega^- = \omega_1 \cdots \omega_n^- = \omega_1 \cdots (\omega_n - 1)$. \color{black}
Notice that we can also define the concatenation of a finite word $\omega = w_1 \ldots w_m$ with a sequence $(d_i)$ as the sequence $\omega(d_i) = w_1\ldots w_m d_1 \ldots$. 
Given a sequence $(d_i) \in \Sigma_2$, the \emph{conjugate of $(d_i)$} is given by the sequence $\overline{(d_i)} = (1-d_i)$. Throughout this paper we will use the lexicographic ordering on sequences and words. We write $\left(c_i\right) \prec\left(d_i\right)$ if there exists $k \in \mathbb{N}$ such that $c_i=d_i$ for $i=$ $1, \ldots, k-1$ and $c_k<d_k$. We write $\left(c_i\right) \preccurlyeq\left(d_i\right)$ if $\left(c_i\right) \prec\left(d_i\right)$ or $\left(c_i\right)=\left(d_i\right)$. We also write $\left(d_i\right) \succ\left(c_i\right)$ if $\left(c_i\right) \prec\left(d_i\right)$, and $\left(d_i\right) \succcurlyeq\left(c_i\right)$ if $\left(c_i\right) \preccurlyeq\left(d_i\right)$.
We say that a sequence $(d_i) \in \Sigma_2$ is:
\begin{enumerate}[($i$)]
	\item \emph{finite} if there exists $K \in \mathbb{N}$ such that 
	\color{black}$d_K=1$, and $d_k = 0$ for all $k> K$ \color{black} and \emph{infinite} otherwise;
	\item \emph{co-finite} if the conjugate of $(d_i)$ is finite and \emph{co-infinite} otherwise;
	\item \emph{doubly-infinite} if $(d_i)$ is infinite and co-infinite.
\end{enumerate}
Observe that the sequences $0^\infty$ and $1^\infty$ are doubly-infinite.

\vspace{0.5em} For any base $q\in(1,2]$, $x=1$ has a largest $q$-expansion with respect to the lexicographic order denoted by ${\beta}(q)$ and a largest infinite $q$-expansion denoted by $\alpha(q)$; see \cite{BaiKom2007, DeVKomLor2016}. Such expansions are  called the \emph{greedy} and \emph{quasi-greedy expansions of $1$ in base $q$} respectively. A number theoretic algorithm to produce such expansions can be found in \cite{BaiKom2007,DeVKom2016}, and its dynamical interpretation can be found in \cite{DajKra2002}.

Now, we give the following:

\begin{definition}
	\label{def:univoqueonebase}
	\leavevmode
	\begin{enumerate}[($i$)]
		\item $\mathcal{U} = \set{q \in (1,2]: 1 \text{ has a unique }q\text{-expansion}}$,
		\item $\overline{\mathcal{U}}$ is the topological closure of $\mathcal{U}$,
		\item $\mathcal{V} = \set{q \in (1,2]: 1 \text{ has a unique doubly-infinite }q\text{-expansion}}$.
	\end{enumerate}
\end{definition}

In \cite[Theorem 2.5, Theorem 3.9]{DeVKomLor2016} the following symbolic characterizations of the sets $\mathcal{U}$, $\overline{\mathcal{U}}$ and $\mathcal{V}$ in terms of the quasi-greedy expansion of $1$ in base $q$ are given.

\begin{theorem}\label{th:lexchar}
	\leavevmode
	\begin{enumerate}[($i$)]
		\item $\mathcal{U}=\set{q\in(1, 2): \overline{\alpha(q)}\prec\sigma^n(\alpha(q))\prec \alpha(q)\quad \text{for} \quad n\in \mathbb{N}}\cup\set{2},$
		\item $\overline{\mathcal{U}}=\set{q\in(1, 2]: \overline{\alpha(q)}\prec \sigma^n(\alpha(q))\preccurlyeq \alpha(q)\quad \text{for} \quad  n\in\mathbb{N}_0},$
		\item $\mathcal{V} =\set{q\in(1, 2]: \overline{\alpha(q)}\preccurlyeq \sigma^n(\alpha(q))\preccurlyeq \alpha(q)\quad \text{for}\quad n\in \mathbb{N}_0}.$
	\end{enumerate}
\end{theorem}

In Theorem \ref{th:lexchar}, $\sigma:\Sigma_{2} \to \Sigma_{2}$ is the \emph{one-sided left shift map}, i.e., $\sigma((d_i)) = (d_{i+1})$; and for each $n \in \mathbb{N}$, $\sigma^n$ is the $n$-fold composition of $\sigma$ with itself.

\vspace{0.5em} Now, let us sum up several relevant results on the topology and the Hausdorff dimension of the sets $\mathcal{U}$, $\overline{\mathcal{U}}$ and $\mathcal{V}$ contained in \cite{KomLor2007,DeVKom2009,KomKonLi2015,DeVKomLor2016,DeVKomLor2022} in the following:

\begin{theorem}
	\label{thm:propertiesUV}
	Denote by $\lambda$ the Lebesgue measure in $\mathbb{R}$ and by $\dim_{\operatorname{H}}F$ the Hausdorff dimension of a subset $F \subset \mathbb{R}$. Then:
	\begin{enumerate}[($i$)]
		\item $\mathcal{U} \subsetneq \overline{\mathcal{U}} \subsetneq \mathcal{V}$.
		\item The sets $\mathcal{U}, \overline{\mathcal{U}}$ and $\mathcal{V}$ are uncountable \color{black} nowhere dense \color{black} \footnotemark and of the Baire's first category.
		\item The set $\mathcal{V}$ is compact.
		\item The set $\overline{\mathcal{U}}$ is a Cantor set.
		\item The sets:
		\begin{enumerate}
			\item $\overline{\mathcal{U}} \setminus \mathcal{U}$ is a countable dense subset of $\overline{\mathcal{U}}$.
			\item $\mathcal{V} \setminus \overline{\mathcal{U}}$ is a countable and discrete subset of $\mathcal{V}$.
		\end{enumerate}
		\item $\lambda(\mathcal{U}) = \lambda(\overline{\mathcal{U}}) = \lambda(\mathcal{V})= 0$.
		\item $\dim_{\operatorname{H}}\mathcal{U} =\dim_{\operatorname{H}}\overline{\mathcal{U}} = \dim_{\operatorname{H}}\mathcal{V} = 1$.
	\end{enumerate}
\end{theorem}
\footnotetext{ A set $A$ is called nowhere dense if its closure has no interior points.}     
It is worth noticing that, in the proof of Theorem \ref{thm:propertiesUV}, the following lexicographic characterization of quasi-greedy expansions of $1$  given in \cite{BaiKom2007} plays a key role.
\begin{proposition}\label{p:lexi of 1}\mbox{}	
	The map
	\[
	\phi:\color{black}[1, 2]\color{black} \to \set{\alpha \in \Sigma_{2} : \sigma^n(\alpha) \preccurlyeq \alpha \text{ for all } n \in \mathbb{N}_0 \text { and } \alpha \text{ is  infinite}}
	\]
	given by $ \phi(q) = \alpha(q)$ is increasing, bijective and continuous from the left. Moreover, the inverse map $\phi^{-1}$ is increasing and continuous.	
\end{proposition}

On the other hand, in a recent core of papers, Li \cite{Li2021}, Neuh\"auserer \cite{Neu2021} and Komornik et al.  \cite{KomLuZou2022} 
extended and studied the definition given in \eqref{eq:q-expansion} to \emph{double-base expansions}.

Given $q_0,q_1 > 1$ we consider  the  alphabet-base system $S=\set{(d_0, q_0), (d_1, q_1)}$. We call a sequence $(i_k)\in \Sigma_{2}$ a \emph{$(q_0,q_1)$-expansion} or a \emph{double-base expansion of $x$} if
\begin{equation}
	\label{eq:q0q1-expansion}
	x = \pi_{q_0,q_1}((i_k)):= \mathop{\sum}\limits_{k=1}^\infty \frac{d_{i_{k}}}{q_{i_1}\cdots q_{i_k}}.
\end{equation}
For a different but  related notion of double-base expansions, we refer the reader to \cite{ChaCisDaj2023}.

It is worth noticing that if $q_0 = q_1$, $d_0=0$ and $d_1=1$ then \eqref{eq:q0q1-expansion} reduces to \eqref{eq:q-expansion}.
It was shown in \cite{KomSteZou2022} that  the system  $S=\set{(d_0, q_0), (d_1, q_1)}$ is isomorphic to  $S=\set{(0, q_0), (1, q_1)}$, hence throughout this paper 
we restrict ourselves to the simpler system $S=\set{(0, q_0), (1, q_1)}$. 

\color{black} To obtain  the most elegant theory,   we consider only the simpler system for $(q_0, q_1)\in (1,\infty)^2$. 
As a consequence of \cite[Theorem 1]{KomLuZou2022}, we have $\pi_{q_0, q_1}\left(0^{\infty}\right)=0, \;\pi_{q_0, q_1}\left(1^{\infty}\right)=1/(q_1-1)$, and then
a number $x$ has a double-base expansion if and only if $x \in [0, 1/(q_1 - 1)]$. 	
Throughout this paper, we always assume that  \begin{equation}
	\label{eq:validcondition}
	(q_0, q_1) \in \set{(q_0,q_1) \in (1,\infty)^2 : q_0+q_1\geq q_0q_1}.
\end{equation}
\color{black}Since if 
$ (q_0, q_1) \in \{(q_0, q_1) \in (1, \infty)^2 : q_0 + q_1 < q_0 q_1 \}, $
then the set of points with an expansion in base $(q_0,q_1)$ is a fractal set, and if a point has an expansion in base $(q_0,q_1)$, then this sequence must be unique. Moreover, neither $q_1/(q_0(q_1-1))-1$ nor $q_0/q_1$ has any double-base expansion, since $q_1/(q_0(q_1-1))-1<0$ and $q_0/q_1>1/(q_1-1)$.
\color{black}


It is a natural question to ask if Theorem \ref{th:lexchar}, Theorem \ref{thm:propertiesUV} and Proposition \ref{p:lexi of 1} can be generalized to the setting of $S=\set{(0, q_0), (1, q_1)}$. It is more complicated, new 
research strategies  are needed.

\color{black}

We recall from Corollary 3 in \cite{KomLuZou2022} that whether a double-base expansion $(x_i)$ of $x\in[0, 1/(q_1-1)]$ is unique  is determined by two key points, which are
\begin{equation*}\label{def:critical}
	\ell_{q_0,q_1} := \frac{q_1}{q_0(q_1-1)} - 1 \qtq{ and } r_{q_0,q_1} := \frac{q_0}{q_1}.
\end{equation*}
In detail, define
$$T_{q_i,i}:\mathbb{R}\rightarrow\mathbb{R},\qtq{}x\rightarrow q_ix-i\qtq{} i\in\{0,1\},$$
a double-base expansion $(x_i)$ of $x\in[0, 1/(q_1-1)]$ is unique if and only if $\pi_{q_0, q_1}(\sigma^n((x_i)))\notin [1/q_1,1/(q_0(q_1-1))]\text{ for all }n\in \mathbb{N}_0$  as shown in Figure \ref{unique expansion}.
\begin{figure}[t]
	\label{unique expansion}
	\centering
	\begin{tikzpicture}[scale=5]
		\draw[gray](0,0)--(1,0)--(1,1)--(0,1)--cycle;
		\draw[dotted](.3334,0)--(.3334,1)(0.5,0)--(0.5,1);
		\draw[thick](0,0)--node[above left]{\scriptsize$T_{q_0,0}$}(0.5,1) (.3334,0)--node[below right]{\scriptsize$T_{q_1,1}$}(1,1);
		\draw(0,0)node[below]{\scriptsize$0$};
		\draw(1,0)node[below]{\scriptsize$\frac{1}{q_1-1}$};
		\draw(0,1)node[left]{\scriptsize$\frac{1}{q_1-1}$};
		\draw(0.3334,0)node[below]{\scriptsize$\frac{1}{q_1}$};
		\draw(0.5,0)node[below]{\scriptsize$\frac{1}{q_0(q_1-1)}$};
	\end{tikzpicture}
	\caption{The maps $T_{q_0,0}$ and $T_{q_1,1}$ generate double-base expansions for $x \in [0,1/(q_1-1)]$ when $q_0 = 2$ and $q_1 = 3/2$.}
\end{figure}

We denote by
\color{black}  $\lambda(q_0,q_1)$ and $\mu(q_0,q_1)$
the lexicographically smallest and the lexicographically smallest co-infinite double-base expansions of $\ell_{q_0, q_1}$, respectively, which are called the \emph{lazy} and \emph{quasi-lazy} double-base expansions of $\ell_{q_0, q_1}$ respectively. In a similar fashion, we denote by
${\beta}(q_0,q_1)$  and $\alpha(q_0,q_1)$ \color{black}
the lexicographically largest and the lexicographically largest infinite double-base expansions of $r_{q_0, q_1}$, which are called the  \emph{greedy} and  \emph{quasi-greedy} double-base expansions of $r_{q_0, q_1}$ respectively. \color{black}  These \color{black} four double-base expansions play an important role in the lexicographic characterization of  quasi-greedy, greedy, quasi-lazy and lazy expansions,  furthermore the unique double-base expansion of points in  $[0, 1/(q_1-1)]$; see \cite{KomLuZou2022} for details.

\color{black}
To state our results, we give some definitions of topology. We endow $\Sigma_2$ with the topology induced by the metric given by:
\begin{equation}\label{eq:distance}
	d((c_i),(e_i)) = \left\{
	\begin{array}{clrr}
		2^{-j} & \text{if} \quad  (c_i) \neq (e_i)&  \text{ where } j = \min\set{i \in \mathbb{N} : c_i \neq e_i},\\
		0 & \text{otherwise.}&\\
	\end{array}
	\right.
\end{equation}
It is a well known fact that $(\Sigma_{2},d)$ is a \color{black} compact \color{black} set. 
During our research, we will consider the set $\Sigma_{2} \times \Sigma_{2}$ with the topology induced by the metric given by:
\begin{equation}\label{eq:distanceSym}
	d_{\Sigma_{2} \times \Sigma_{2}}(((c_i),(e_i)),((c'_i),(e'_i))) = \mathop{\max}\set{d((c_i),(c'_i)), d((e_i),(e'_i))}.
\end{equation}


Consider $\mathbb{R}^2$ with the \emph{$\max$ topology}, i.e., the topology of $\mathbb{R}^2$ given by the metric:
\begin{equation}\label{eq:distanceR2}
	d_{\mathbb{R}^2}((q_0,q_1), (q'_0,q'_1)) = \mathop{\max}\set{|q_0-q'_0|,|q_1-q'_1|}.	
\end{equation}
It is a well known fact that the topology induced by $d_{\mathbb{R}^2}$ is equivalent to the usual euclidean topology in $\mathbb{R}^2$ as well as the product topology in $\mathbb{R}^2$.
\color{black}

We now proceed with the following definitions. 
\begin{definition}
	\label{def:q0q1sets}
	\leavevmode
	\begin{enumerate}[$(i)$]
		\item $\mathcal{U}_2:= \set{(q_0,q_1) \in (1,\infty)^2: \ell_{q_0,q_1} \text{ and } r_{q_0,q_1} \text{ have a unique  double-base   expansion}}$,
		\item $\overline{\mathcal{U}_2}$ is the topological closure of $\mathcal{U}_2$ with respect to the topology  \eqref{eq:distanceR2},
		\item $\mathcal{V}_2:= \{(q_0,q_1) \in (1, \infty)^2 : \sigma^m(\mu(q_0,q_1)) \preccurlyeq \alpha(q_0,q_1),\sigma^n(\alpha(q_0,q_1)) \succcurlyeq \mu(q_0,q_1)$ for all $m,n \in \mathbb{N}_0\} 
		$\footnote{In our sequel paper, we have proved that this definition is equivalent to $$\mathcal{V}_2:= \set{(q_0,q_1) \in (1,\infty)^2: \ell_{q_0,q_1} \text{ and } r_{q_0,q_1} \text{ have a unique doubly-infinite double-base   expansion}}.$$}.
	\end{enumerate}
\end{definition}

Following we define some corresponding set of sequences. To do this, let us introduce the following symbolic sets. Let
\begin{equation*}
	\label{eq:sig0sig1}
	\Sigma_{2}^0: = \set{(d_i) \in \Sigma_{2} : d_1 = 0}  \qtq{ and }  \Sigma_{2}^1 := \set{(d_i) \in \Sigma_{2} : d_1 = 1}.
\end{equation*}

\begin{definition}	\label{def:vb2}
	\leavevmode
	\begin{enumerate}[$(i)$]
		\item $\mathcal{U}'_2:=\{(\mu, \alpha)\in \Sigma_2^0 \times \Sigma_2^1:\mu \prec \sigma^m(\mu) \prec \alpha,\mu \prec \sigma^n(\alpha) \prec \alpha  \text{ for all } m,n \in \mathbb{N}\},$
		\item $\overline{\mathcal{U}'_2}$ is the topological closure of $\mathcal{U}'_2$ with respect to the topology \eqref{eq:distanceSym},
		\item $\mathcal{V}'_2:=\{(\mu, \alpha)\in \Sigma_2^0 \times \Sigma_2^1:\mu \preccurlyeq \sigma^m(\mu) \preccurlyeq \alpha,\mu\preccurlyeq \sigma^n(\alpha) \preccurlyeq \alpha  \text{ for all } m,n \in \mathbb{N}_0\}$.
	\end{enumerate}
\end{definition}

Now we are ready to state our first main result, which generalizes Proposition \ref{p:lexi of 1} in the double-base system. Define
\color{black}
\begin{equation}
	\label{eq:AB}
	\begin{split}
		\mathcal{B} &:= \set{(q_0,q_1) \in (1,\infty)^2 : q_0+q_1 > q_0q_1}, \\
		\mathcal{C} &:= \set{(q_0,q_1) \in (1,\infty)^2 : q_0 + q_1 = q_0 q_1}  , \\
		\mathcal{C}' &:= \left(\set{0^\infty, x_1 \dots x_m 1^\infty} \times \Sigma_2^1\right) \cup \left(\Sigma_2^0 \times \set{1^\infty, y_1 \dots y_n 0^\infty}\right) \\ 
		&\text{ with } x_i, y_j \in \set{0,1}, x_m=0, y_n=1, \text{ and } m, n \in \mathbb{N}, \\
		\mathcal{B}' &:= \set{(\mu, \alpha) \in \left(\Sigma_2^0 \times \Sigma_2^1\right) \setminus \mathcal{C}' : \mu \preccurlyeq \sigma^m(\mu) \text{ and } \sigma^n(\alpha) \preccurlyeq \alpha \text{ for all } m, n \in \mathbb{N}_0}.
	\end{split}
\end{equation}

Figure \ref{fig:figure2} represents the visualization of the set \( \mathcal{B} \) and its upper boundary \( \mathcal{C} \), providing a clear demonstration of the defined region and its constraints.
\color{black}
\begin{figure}[htbp]
	\centering
	\includegraphics[width=5cm]{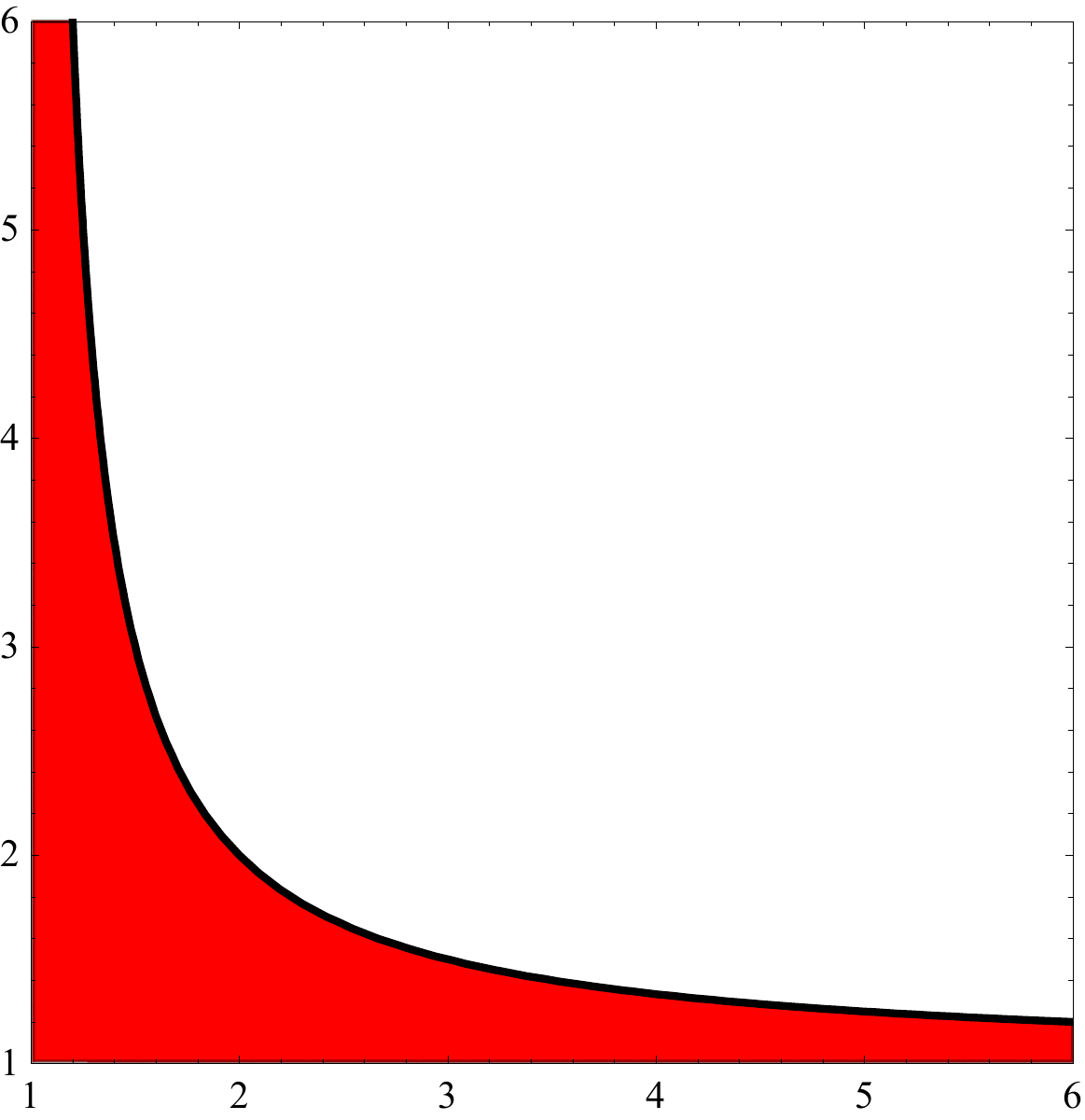} 
	\caption{The red region represents $\mathcal{B}$, and the black curve represents $\mathcal{C}$.}
	\label{fig:figure2}
\end{figure}

\begin{main}
	\label{t:main1}
	The map $\Phi: \mathcal{B} \to \mathcal{B}'$ defined by $\Phi((q_0, q_1)) = (\mu(q_0,q_1), \alpha(q_0,q_1))$ is bijective. Moreover, $\Phi^{-1}$ is continuous in $\mathcal{B}'$.
\end{main}

\begin{remark}\label{r:monotonic}
	\begin{enumerate}[\upshape(i)]
		\item We  have $q_0+q_1=q_0q_1\Leftrightarrow\mu(q_0,q_1)= 0^\infty\Leftrightarrow\alpha(q_0,q_1)= 1^\infty.$
		Hence, to say $\Phi$ is bijective, we assume that $(q_0, q_1)\in\mathcal{B} $. 
		\item $(\mu,\alpha)\in\mathcal{B}'$ implies \color{black} 
		$\mu$ is co-infinite and $\alpha$ is infinite. \color{black}  Furthermore,  we have $\Phi^{-1}((\mu, \alpha))=(q_0(\mu, \alpha), q_1(\mu, \alpha))$ and $(q_0(\mu, \alpha), q_1(\mu, \alpha)) \in\mathcal{B}. $
		\item\cite[Lemma 3.3]{KomSteZou2022} showed that the map $\Phi$ is monotonic under some condition, i.e.,
		for $q'_0 \ge q_0$, $q'_1 \ge q_1$, with $q'_0+q'_1 \ge q'_0 q'_1 > q_0 q_1$, we have
		\begin{equation*}
			\mu(q'_0,q'_1)\prec\mu(q_0,q_1)  \quad \text{and} \quad \alpha(q'_0,q'_1) \succ\alpha(q_0,q_1).                      
		\end{equation*}                                                   However, the inverse is not right. For example $\alpha(q_0,q_1)=(11100)^\infty$, $\mu(q_0,q_1)=(00101)^\infty$,
		$\alpha(q'_0,q'_1)=(1110)^\infty$, $\mu(q'_0,q'_1)=(00100101)^\infty$. But we  have $q'_0>q_0$ and $q'_1<q_1$.          \color{black} \item   $\Phi$ is discontinuous in $\mathcal{B}$. Let $q_0 = q_1 = (\sqrt{5} + 1)/2$. Then $\mu(q_0, q_1) = (01)^\infty, \quad \alpha(q_0, q_1) = (10)^\infty.$
		If $(q'_0, q'_1)$ is sufficiently close to $(q_0, q_1)$ and satisfies 
		$q'_0, q'_1 > (\sqrt{5}+1)/2$, then by (iii), we have $\mu(q'_0, q'_1)$ and $\alpha(q'_0, q'_1)$ begins with $00$ and $11$, respectively. This implies $\Phi$ is discontinuous in $\mathcal{B}$. \color{black}
	\end{enumerate}
\end{remark}

The following theorem provides information on the topology of the sets $\mathcal{U}_2'$, $\overline{\mathcal{U}'_2}$ and $\mathcal{V}'_2$.
\begin{main}\label{t:main2}
	The following statements hold:
	\begin{enumerate}[$(i)$]
		\item $\mathcal{U}'_2 \subsetneq \overline{\mathcal{U}'_2} \subsetneq \mathcal{V}'_2$.
		\item $(\mu, \alpha)\in\overline{\mathcal{U}'_2}$ if and only if $(\mu, \alpha)\in\mathcal{V}'_2$ and there are no	$m, n \in\mathbb{N}$ such that
		both
		$\sigma^{m}(\mu)=\alpha$ and $\sigma^{n}(\alpha)=\mu$
		hold.
		{	\item  $(\mu, \alpha)\in\mathcal{V}'_2\setminus\overline{\mathcal{U}'_2}$ if and only if \color{black}  $(\mu, \alpha)\in\mathcal{V}'_2$ and  there exist $m, n \in \mathbb{N}$ such that	both
			$\sigma^{m}(\mu)=\alpha$ and $\sigma^{n}(\alpha)=\mu$ hold.\color{black} }
		
		\item $\mathcal{V}'_2$ is compact with respect to the topology given in \eqref{eq:distanceSym}.
		\item  $\overline{\mathcal{U}'_2}$ is a Cantor set.
		\item  $\overline{\mathcal{U}'_2} \setminus \mathcal{U}'_2$ is dense in $\overline{\mathcal{U}'_2}$; $\mathcal{V}'_2 \setminus \overline{\mathcal{U}'_2}$ is countable and dense in $\mathcal{V}'_2$.
		\item $\overline{\mathcal{U}'_2} \setminus \mathcal{U}'_2$ is 	 uncountable.
	\end{enumerate}
\end{main}

The following Theorem \ref{t:main3} describes the connections between  the sets $\mathcal{U}_2$, $\overline{\mathcal{U}_2}$ and $\mathcal{V}_2$, and their corresponding  sets of sequences $\mathcal{U}'_2$, $\overline{\mathcal{U}'_2}$ and $\mathcal{V}'_2$.

\begin{main}\label{t:main3}
	\quad
	\begin{enumerate}[($i$)]
		\item  $ \mathcal{U}_2=\Phi^{-1}(\mathcal{U}'_2)\cup \mathcal{C}.$
		\item
		$\mathcal{V}_2=\Phi^{-1}(\mathcal{V}'_2\setminus\mathcal{C}')\cup \mathcal{C}.$
		\item
		$\overline{\mathcal{U}_2}=\Phi^{-1}(\overline{\mathcal{U}'_2}\setminus\mathcal{C}')\cup \mathcal{C}.$
	\end{enumerate}
\end{main}

\begin{main}\label{t:main4}\color{black}
	The following statements hold:
	\begin{enumerate}[($i$)]
		\item $\mathcal{U}_2$, $\overline{\mathcal{U}_2}$  and $\mathcal{V}_2$ are unbounded \color{black} nowhere dense \color{black} sets and of the Baire's first category.
		\item $\mathcal{U}_2 \subsetneq \overline{\mathcal{U}_2} \subsetneq \mathcal{V}_2$.
		\item The set $\mathcal{V}_2$ is closed, and the set 
		$\mathcal{U}_2$ is  \color{black} a countable union of Cantor sets.\color{black}
		\item The set:
		\begin{enumerate}
			\item $\overline{\mathcal{U}_2} \setminus \mathcal{U}_2$ is  dense  in $\overline{\mathcal{U}_2}$.
			\item  $\mathcal{V}_2 \setminus \overline{\mathcal{U}_2}$ is a  countable discrete set, dense in $\mathcal{V}_2\setminus\mathcal{C}$.
		\end{enumerate}
		\item $\dim_{\operatorname{H}}\mathcal{U}_2 = \dim_{\operatorname{H}}\overline{\mathcal{U}_2}= \dim_{\operatorname{H}}\mathcal{V}_2 = 2$ and $\dim_{\operatorname{H}}\overline{\mathcal{U}_2} \setminus \mathcal{U}_2 \geq 1 $.
	\end{enumerate}
\end{main}

Our research is arranged as follows. In Section \ref{sec:preliminaries} we recall some relevant
definitions and results \color{black} on \color{black} expansions in double-base. \color{black}
Theorem \ref{t:main1} is proved in Section \ref{sec:3}. Section \ref{sec:4} is devoted to calculating the Hausdorff dimension of sets $\overline{\mathcal{U}_2} \setminus \mathcal{U}_2 $ and $\mathcal{U}_2$ in Theorem \ref{t:main4}. Theorem \ref{t:main2} is established in Sections \ref{sec:5} and \ref{sec:6}. Besides, Section \ref{sec:6} contains  the proof of Theorem \ref{t:main3} and  \color{black} Theorem \ref{t:main4} (i)-(iv) . \color{black} 

\vspace{1em}

\section{Preliminaries}
\label{sec:preliminaries}
We  recall from  \cite[Proposition 13]{KomLuZou2022} the relation between the lazy and the quasi-lazy (the greedy and the quasi-greedy ) expansions as follows.
\begin{equation}\label{eq:qlazy}
	\mu({q_0, q_1})  = \left\{
	\begin{array}{clrr}
		\lambda({q_0, q_1}) & \text{ if }\lambda({q_0, q_1})   \text{ is co-infinite,}\\
		(\lambda_1 \ldots \lambda_m^+)^\infty & \text{ otherwise, and  } m = \mathop{\max}\set{i \in \mathbb{N}: \lambda_i =0}\\
	\end{array}
	\right.
\end{equation}
and
\begin{equation}\label{eq:qgreedy2}
	\alpha({q_0, q_1})   = \left\{
	\begin{array}{clrr}
		\beta({q_0, q_1}) & \text{ if }\beta({q_0, q_1})   \text{ is infinite,}\\
		(\beta_1 \ldots \beta_n^-)^\infty & \text{otherwise, and } n = \mathop{\max}\set{i \in \mathbb{N}: \beta_i=1}.\\
	\end{array}
	\right.
\end{equation}

The \emph {greedy (or quasi-greedy) algorithm} provides a (or an infinite) double-base  expansion of $x\in [0, 1/(q_1-1)]$ for each pair $(q_0, q_1)\in\mathcal{B}$, which is defined recursively as follows: if for some positive integer $n$ the digits $\beta_1,\cdots, \beta_{n-1}$ (or $\alpha_1,\cdots,\alpha_{n-1}$) are already defined (no condition if $n=1$), then $\beta_{n}$ (or $\alpha_n$) is the largest digit in $\{0,1\}$ such that $$ \mathop{\sum}\limits_{k=1}^n \frac{\beta_{k}}{q_{\beta_{1}}\cdots q_{\beta_{k}}}\le  x \left(\text { or }  \mathop{\sum}\limits_{k=1}^n \frac{\alpha_{k}}{q_{\alpha_{1}}\cdots q_{\alpha_{k}}}< x\right)$$ holds. The resulting \emph {greedy (or quasi-greedy) expansion} is clearly the lexicographically largest (or largest infinite) double-base expansion.

Similarly, 
the \emph {lazy (or quasi-lazy) algorithm} provides a (or a co-infinite) double-base  expansion of $x\in [0, 1/(q_1-1)]$ for each pair $(q_0, q_1)\in\mathcal{B}$, which is defined recursively as follows: if for some positive integer $m$ the digits $\lambda_1,\cdots, \lambda_{m-1}$ (or $\mu_1,\cdots,\mu_{m-1}$) are already defined (no condition if $m=1$), then $\lambda_{m}$ (or $\mu_m$) is the smallest digit in $\{0,1\}$ such that $$ \mathop{\sum}\limits_{k=1}^m \frac{\lambda_{k}}{q_{\lambda_{1}}\cdots q_{\lambda_{k}}}+\frac{1}{q_{\lambda_{1}}\cdots q_{\lambda_{m}}(q_1-1)}\ge x \left(\text { or } \mathop{\sum}\limits_{k=1}^m\frac{{\mu}_{k}}{q_{{\mu}_{1}}\cdots q_{{\mu}_{k}}}+\frac{1}{q_{{\mu}_{1}}\cdots q_{{\mu}_{m}}(q_1-1)}> x\right)$$ holds. The resulting \emph {lazy (or quasi-lazy) expansion} is clearly the lexicographically smallest (or smallest co-infinite) double-base expansion.
\color{black}

A double-base expansion is unique if and only it is greedy and lazy at the same time. We recall from  \cite[Corollary 3]{KomLuZou2022} that
a double-base expansion $(x_i)$ of point $x$ belonging to $[0, 1/(q_1-1)]$
is unique if and only if
\begin{equation}\label{e:unique-expanison}
	\left\{
	\begin{array}{clrr}
		\mu(q_0, q_1)\prec\sigma^m((x_i)) & \text{ whenever }x_m=1, \\
		\sigma^m((x_i))\prec\alpha (q_0, q_1) & \text{ whenever } x_m=0.\\
	\end{array}
	\right.
\end{equation}



\vspace{1em}

\section{Proof of Theorem \ref{t:main1}}
\label{sec:3}
This section is devoted to prove Theorem \ref{t:main1} that we shall use very frequently in the sequel.
First  we will recall a
number of technical results from \cite {KomSteZou2022}  that will assist in our proof of Theorem \ref{t:main1}.
Consider the two equations   $\pi_{q_0, q_1}(\mu)=\ell_{q_0, q_1}$ and $\pi_{q_0, q_1}(\alpha)=r_{q_0, q_1}$.
Define
\begin{equation*}
	\begin{split}
		&\tilde{\pi}_{q_0,q_1}(x_1x_2\cdots) := \sum_{k=1}^\infty \frac{1-x_k}{q_{x_1}q_{x_2}\cdots q_{x_k}}.
	\end{split}
\end{equation*}
Then equation $\pi_{q_0, q_1}(\mu)=\ell_{q_0, q_1}$ can be rewritten as $\tilde{\pi}_{q_0, q_1}(\mu)=q_1/q_0$.
\color{black} Define \color{black}
\begin{equation*}
	f_{\alpha}(q_0, q_1) :=\pi_{q_0,q_1}(\alpha)-q_0/q_1, \quad \tilde{f}_{\mu}(q_0,q_1) :=\tilde{\pi}_{q_0,q_1}(\mu)-q_1/q_0.
\end{equation*}
In the following lemma we recall some properties of the functions $f_{\alpha}(q_0, q_1)$ and $\tilde{f}_{\mu}(q_0,q_1)$ from \cite{KomSteZou2022}.
\begin{lemma} \cite [Lemmas 4.1 and 4.2] {KomSteZou2022}\label{l:functions} \color{black} Let $(\mu,\alpha)\in \mathcal{B}' $.
	\begin{enumerate}[\upshape (i)]
		\item\color{black}For every $q_0 > 1$, the equations $f_{\alpha}(q_0, q_1) = 0$ and $\tilde{f}_{\mu}(q_0, q_1) = 0$ have unique solutions $q_1$, denoted by $g_{\alpha}(q_0)$ and $\tilde{g}_{\mu}(q_0)$, respectively. \color{black} And there is a unique $q_0 > 1$ depending on $\alpha$ such that $f_{\alpha}(q_0,1) = 0$; we denote this $q_0$ by~$q_{\alpha}$.
		\item ~$g_{\alpha}(q_0)$ is continuous in $q_0>1$ and strictly decreasing in $(1, q_{\alpha})$; $\tilde{g}_{\mu}(q_0)$  is continuous  and strictly decreasing in $q_0>1$.
		\item $
		\lim_{q_0\to 1} (g_{\alpha}(q_0)-\tilde{g}_{\mu}(q_0)) = \infty \quad \text{and} \quad
		g_{\alpha}(q_0) - \tilde{g}_{\mu}(q_0) < 0 \quad \text{for all}\ q_0 \ge q_{\alpha}.
		$
		\item If \color{black} $g_{\alpha}(q_0)-\tilde{g}_{\mu}(q_0)=0$ have a unique solution \color{black}, denoted by $q_0(\mu, \alpha)$. Then  $$1<q_0(\mu, \alpha)<q_\alpha \text{ and } q_1(\mu, \alpha)=g_{\alpha}(q_0(\mu, \alpha))=\tilde{g}_{\mu}(q_0(\mu, \alpha)).$$
		\item If $g_\alpha(x)-\tilde{g}_\mu(x)>0$ then  $x<q_0(\mu, \alpha)$. If $g_\alpha(x)-\tilde{g}_\mu(x)<0$ then  $x>q_0(\mu, \alpha)$.
	\end{enumerate}
\end{lemma}

\begin{lemma}\label{l:43KomSteZou2022}\cite [Lemma 4.3]{KomSteZou2022}
	Let $q_0  > 1$.
	Then the map
	\begin{equation*}
		\mathcal{G}:\, \alpha \to \big(1,\tfrac{q_0}{q_0-1}\big), \quad \alpha \mapsto g_{\alpha}(q_0),
	\end{equation*}
	is a continuous order-preserving bijection, and the map
	\begin{equation*}
		\tilde{\mathcal{G}}:\, \mu\to \big(1,\tfrac{q_0}{q_0-1}\big), \quad \mu \mapsto \tilde{g}_{\mu}(q_0),
	\end{equation*}
	is a continuous order-reversing bijection.
\end{lemma}

\color{black}
The following Proposition \ref{l: bijective-map} plays a critical role in the proof of Theorem \ref{t:main1}, which generalizes the result of Lemma 4.4 in \cite{KomSteZou2022}. \color{black}
\begin{proposition}\label{l: bijective-map}
	The map $\Phi: \mathcal{B} \to \mathcal{B}'$ defined by $\Phi((q_0, q_1)) = (\mu(q_0,q_1), \alpha(q_0,q_1))$ is bijective.
\end{proposition}
\begin{proof}
	We show that there exists a unique pair $(q_0(\mu, \alpha), q_1(\mu, \alpha))\in\mathcal{B}$  if  $(\mu, \alpha)\in\mathcal{B}'$, and then such that  $\alpha$ is the quasi-greedy expansion of $r_{q_0, q_1}$ and  $\mu$ is the quasi-lazy expansion of $\ell_{q_0, q_1}$. Since $\alpha\prec1^\infty$, we have $\pi_{q_0, q_1}(\alpha)=q_0/q_1<\pi_{q_0, q_1}(1^\infty)=1/(q_1-1)$, this implies  $q_0+q_1>q_0q_1$. 
	
	It follows from  Lemma \ref{l:functions} that  the two equations $	f_{\alpha}(q_0, q_1) =0$ and $\tilde{f}_{\mu}(q_0,q_1)=0$ have unique solutions $g_{\alpha}(q_0)$ and ~$\tilde{g}_{\mu}(q_0)$, respectively.  Next,  we  show that the equation ~$g_{\alpha}(q_0)= \tilde{g}_{\mu}(q_0)$ has a unique solution $q_0 > 1$. It is shown in Lemma \ref{l:functions} that $g_{\alpha}(x)$ and $\tilde{g}_{\mu}(x)$ are continuous functions with
	\begin{equation*}
		\lim_{q_0\to 1} (g_{\alpha}(q_0)-\tilde{g}_{\mu}(q_0)) = \infty \quad \text{and} \quad
		g_{\alpha}(q_0) - \tilde{g}_{\mu}(q_0) < 0 \quad \text{for all}\ q_0 \ge q_{\alpha}.
	\end{equation*}
	Moreover, the functions are differentiable on $(1,q_{\alpha})$ with
	\begin{equation*}
		g_{\alpha}'(x) = - \frac{D_1 f_{\alpha}(x, g_{\alpha}(x))}{D_2 f_{\alpha}(x, g_{\alpha}(x))}, \quad 
		\tilde{g}_{\mu}'(x) = - \frac{D_1 \tilde{f}_{\mu}(x, \tilde{g}_{\mu}(x))}{D_2 \tilde{f}_{\mu}(x, \tilde{g}_{\mu}(x))}.
	\end{equation*}
	Therefore, it suffices to show that $g_{\alpha}'(x) < \tilde{g}_{\mu}'(x)$ whenever $g_{\alpha}(x)=\tilde{g}_{\mu}(x)$, i.e.,
	\begin{equation} \label{e:partial1}
		D_1 f_{\alpha}(x,y) D_2 \tilde{f}_{\mu}(x,y) - D_2 f_{\alpha}(x,y) D_1 \tilde{f}_{\mu}(x,y) > 0
	\end{equation}
	whenever $f_{\alpha}(x,y) = 0 = \tilde{f}_{\mu}(x,y)$, $1 < x < q_{\alpha}$, $y > 1$.
	
	Write
	\begin{equation*}
		\alpha = 10^{n_1}10^{n_2-n_1}10^{n_3-n_2}\cdots \quad \text{and} \quad \mu = 01^{\tilde{n}_1}01^{\tilde{n}_2-\tilde{n}_1}01^{\tilde{n}_3-\tilde{n}_2}\cdots
	\end{equation*}
	with $0 \le n_1 \le n_2 \le \cdots$, $0 \le \tilde{n}_1 \le \tilde{n}_2 \le \cdots$.
	Then we have
	\begin{equation*}
		f_{\alpha}(x,y) = \frac{1}{y}\left(1-x + \sum_{k=1}^\infty \frac{1}{y^kx^{n_k}}\right),
		\quad \tilde{f}_{\mu}(x,y) = \frac{1}{x}\left(1-y+ \sum_{k=1}^\infty \frac{1}{x^ky^{\tilde{n}_k}}\right),
	\end{equation*}
	hence \eqref{e:partial1} means that
	\begin{equation*} \label{e:partial2}
		\bigg(1 + \sum_{k=1}^\infty \frac{n_k}{y^kx^{n_k+1}}\bigg) \bigg(1 + \sum_{\ell=1}^\infty \frac{\tilde{n}_\ell}{x^\ell y^{\tilde{n}_\ell+1}}\bigg) - \bigg(\frac{1}{y}-\frac{x}{y}+\sum_{k=1}^\infty \frac{k+1}{y^{k+1}x^{n_k}}\bigg)\, \bigg(\frac{1}{x}-\frac{y}{x}+\sum_{\ell=1}^\infty \frac{\ell+1}{x^{\ell+1}y^{\tilde{n}_\ell}} \bigg)> 0.
	\end{equation*}
	By noting that $f_{\alpha}(x,y) = 0 = \tilde{f}_{\mu}(x,y)$, this inequality reduces to
	\begin{equation} \label{e:partial4}
		\bigg(1 + \sum_{k=1}^\infty \frac{n_k}{y^kx^{n_k+1}}\bigg) \bigg(1 + \sum_{\ell=1}^\infty \frac{\tilde{n}_\ell}{x^\ell y^{\tilde{n}_\ell+1}}\bigg) - \sum_{k=1}^\infty \frac{k}{y^{k+1}x^{n_k}}\, \sum_{\ell=1}^\infty \frac{\ell}{x^{\ell+1}y^{\tilde{n}_\ell}} > 0.
	\end{equation}
	Applying $f_{\alpha}(x,y) = 0 = \tilde{f}_{\mu}(x,y)$ again, we have $$\frac{1}{x-1} \sum_{k=1}^\infty \frac{1}{y^kx^{n_k}} = 1 = \frac{1}{y-1} \sum_{\ell=1}^\infty \frac{1}{x^\ell y^{\tilde{n}_\ell}}.$$
	Inserting this into \eqref{e:partial4} and multiplying by $xy$ gives the inequality	
	\begin{equation} \label{e:partial5}
		\left( 1 + \sum_{k=1}^\infty \frac{n_k}{y^k x^{n_k + 1}} \right) 
		\left( 1 + \sum_{\ell=1}^\infty \frac{\tilde{n}_\ell}{x^\ell y^{\tilde{n}_\ell + 1}} \right) 
		- \sum_{k=1}^\infty \frac{k}{y^{k+1} x^{n_k}} \sum_{\ell=1}^\infty \frac{\ell}{x^{\ell+1} y^{\tilde{n}_\ell}} 
		> 0.
	\end{equation}
	\eqref{e:partial5}  follows from \color{black} Proposition \ref{p1} (shown in below) 
	since $n_k\ge 0$ , $n_\ell\ge 0$ for all $k\ge 1$ and $\ell\ge 1$.
	
	Now we turn to prove the map is  injective.  For $(q_0, q_1)\in \mathcal{B}$, using the quasi-greedy and quasi-lazy  algorithm in \cite{KomLuZou2022}, we  obtain the unique quasi-greedy expansion $\alpha(q_0, q_1)$ of $r_{q_0, q_1}$ and the unique quasi-lazy expansion $\mu(q_0, q_1)$ of $\ell_{q_0, q_1}$.  In other words,	$ \mu(q_0, q_1) \preccurlyeq\sigma^{i}(\mu(q_0, q_1) )$ and  $ \sigma^{j}(\alpha(q_0, q_1))\preccurlyeq\alpha(q_0, q_1)$ for all $i, j\ge 0$. Furthermore, $\mu(q_0, q_1)\neq 0^\infty$ and $\alpha(q_0, q_1)\neq 1^\infty$ by Remark \ref{r:monotonic} (i). Hence, $(\mu(q_0, q_1), \alpha(q_0, q_1))\in\mathcal{B}'$.
	\color{black} If there exists another $(q'_0, q'_1)\in \mathcal{B}$ such that $(\mu(q'_0, q'_1), \alpha(q'_0, q'_1))=(\mu(q_0, q_1), \alpha(q_0, q_1))\in\mathcal{B}'$. This contradicts the fact that there  exists a unique pair $(q_0(\mu, \alpha), q_1(\mu, \alpha))\in\mathcal{B}$  if  $(\mu, \alpha)\in\mathcal{B}'$. \color{black}
\end{proof}

\begin{lemma}\label{p:inverse-map}
	$\Phi^{-1}$ is continuous in $\mathcal{B}'$.
\end{lemma}
\begin{proof}
	Fix $(\mu, \alpha)\in \mathcal{B}'$ and let $(\mu^n, \alpha^n)\in \mathcal{B}'$ satisfy
	$$d_{\Sigma_{2} \times \Sigma_{2}}((\mu, \alpha),(\mu^n, \alpha^n))<\delta_1:=\delta_1(\mu^n, \alpha^n).$$	
	\color{black} Here, $\delta_1$ is a arbitrary small positive number depending on $(\mu^n, \alpha^n)$. \color{black} We have
	$$\Phi^{-1}((\mu^n, \alpha^n))=(q_0(\mu^n, \alpha^n), q_1(\mu^n, \alpha^n))\text{ with } 1<q_0(\mu^n, \alpha^n)<q_{\alpha^n}. $$
	Then it follows from Lemma \ref{l:43KomSteZou2022} that $$|\tilde{g}_\mu(x)-\tilde{g}_{\mu^n}(x)|<\delta_2(\delta_1),\;\; |g_\alpha(x)-g_{\alpha^n}(x)|<\delta_3(\delta_1)\text{ for any } x>1.$$
	\color{black} 
	$\delta_2, \delta_3 \to 0$ as  $\delta_1\to 0$. \color{black}
	This together with $\tilde{g}_{\mu}(q_0(\mu, \alpha))=g_{\alpha}(q_0(\mu, \alpha))$ and $\tilde{g}_{\mu^n}(q_0(\mu^n, \alpha^n))=g_{\alpha^n}(q_0(\mu^n, \alpha^n))$ leads to
	\begin{equation}\label{e:43}
		\begin{split}
			&|g_{\alpha}(q_0(\mu^n, \alpha^n))-\tilde{g}_{\mu}(q_0(\mu^n, \alpha^n))|\\
			=&|g_{\alpha}(q_0(\mu^n, \alpha^n))-g_{\alpha^n}(q_0(\mu^n, \alpha^n))+g_{\alpha^n}(q_0(\mu^n, \alpha^n))-\tilde{g}_{\mu^n}(q_0(\mu^n, \alpha^n))\\&+\tilde{g}_{\mu^n}(q_0(\mu^n, \alpha^n))-\tilde{g}_{\mu}(q_0(\mu^n, \alpha^n))|
			<\delta_3(\delta_1)+\delta_2(\delta_1)=\delta_4(\delta_1).\\
		\end{split}
	\end{equation}
	Similarly, $\delta_4\to 0$ as  $\delta_1\to 0$. In the proof of Lemma \ref{l: bijective-map}, we have shown that the function $g_\alpha(x)-\tilde{g}_\mu(x)$ is continuous in $x$ and  have a unique solution $q_0(\mu, \alpha)$ in $(1, q_{\alpha})$. We also note that both $q_0(\mu, \alpha)$ and  $q_0(\mu^n, \alpha^n)$ are bounded.
	These facts together with \eqref{e:43} imply that
	$$|q_0(\mu, \alpha)-q_0(\mu^n, \alpha^n)|<\delta_5(\delta_4).$$ Here $\delta_5\to 0$ as  $\delta_4\to 0$. Hence we have
	\begin{equation*}
		\begin{split}
			|q_1(\mu, \alpha)-q_1(\mu^n, \alpha^n)|&=|g_\alpha(q_0(\mu, \alpha))-g_{\alpha^n}(q_0(\mu^n, \alpha^n))|\\&=|g_\mu(q_0(\mu, \alpha))-g_\alpha(q_0(\mu^n,\alpha^n))+g_\alpha(q_0(\mu^n, \alpha^n))-g_{\alpha^n}(q_0(\mu^n, \alpha^n))|\\
			&\leq|g_\alpha(q_0(\mu, \alpha))-g_{\alpha}(q_0(\mu^n, \alpha^n))|+|g_\alpha(q_0(\mu^n, \alpha^n))-g_{\alpha^n}(q_0(\mu^n, \alpha^n))|\\
			&<\delta_6(\delta_5)+\delta_3(\delta_1).\\
		\end{split}
	\end{equation*}
	Here $\delta_6\to 0$ as  $\delta_5\to 0$.  \color{black}
\end{proof}
\begin{proof}[Proof of Theorem \ref{t:main1} ] 
	It follows from Proposition \ref{l: bijective-map} and Lemma \ref{p:inverse-map}.
\end{proof}
\begin{proposition}\label{p1}
	Fix two real numbers $x,y>1$.
	For any two non-decreasing sequences $(n_k)$ and $\tilde n_{\ell}$ of real numbers, the following inequality holds:
	\begin{equation*}
		S\bj{(n_k), (\tilde n_{\ell})}:=\sum_{k=1}^{\infty}\sum_{\ell=1}^{\infty}\frac{\frac{xy}{(x-1)(y-1)}-k\ell}{x^{n_k+\ell}y^{k+\tilde n_{\ell}}}\ge 0.
	\end{equation*}
	Moreover, equality holds if and only if the sequences $(n_k)$ and $(\tilde n_{\ell})$  are constant.
\end{proposition}

The proof of the above proposition is based on the following two lemmas:

\begin{lemma}\label{l1}
	If the sequences $(n_k)$ and $\tilde n_{\ell}$ are constant, then $S\bj{(n_k), (\tilde n_{\ell})}=0$.
\end{lemma}

\begin{proof}
	We recall for all $t>1$ the elementary relation
	\begin{equation*}
		\sum_{n=1}^{\infty}\frac{n}{t^n}
		=\sum_{n=1}^{\infty}\bj{\sum_{m=1}^n\frac{1}{t^n}}
		=\sum_{m=1}^{\infty}\bj{\sum_{n=m}^{\infty}\frac{1}{t^n}}
		=\sum_{m=1}^{\infty}\frac{1}{t^{m-1}(t-1)}
		=\frac{t}{(t-1)^2}.
	\end{equation*}
	
	Using this and setting $z:=\frac{xy}{(x-1)(y-1)}(>1)$ for brevity, we obtain that 
	\begin{align*}
		S\bj{(n_k), (\tilde n_{\ell})}&=\frac{1}{x^{n_1}y^{\tilde n_1}}\bj{z\sum_{k=1}^{\infty}\frac{1}{y^k}\sum_{\ell=1}^{\infty}\frac{1}{x^{\ell}}
			-\sum_{k=1}^{\infty}\frac{k}{y^k}\sum_{\ell=1}^{\infty}\frac{\ell}{x^{\ell}}}\\
		&=\frac{z}{x^{n_1}y^{\tilde n_1}}\bj{\frac{z}{(x-1)(y-1)}-\frac{yx}{(y-1)^2(x-1)^2}}=0.\qedhere
	\end{align*}
\end{proof}

\begin{lemma}\label{l2}
	Assume that
	\begin{equation}\label{1}
		S\bj{(n_k), (\tilde n_{\ell})}\le 0
	\end{equation} 
	for some non-decreasing sequences  $(n_k), (\tilde n_{\ell})$, and that $n_{k'}<n_{k'+1}$ for some $k'$. 
	Let us introduce a new non-decreasing sequence $(\hat n_k)$ by the formula
	\begin{equation*}
		\hat n_k:=
		\begin{cases}
			n_k&\text{if }k\le k',\\
			n_k-(n_{k'+1}-n_{k'})&\text{if }k>k'.
		\end{cases}
	\end{equation*}
	Then we have $\hat n_{k'+1}=\hat n_{k'}$, and
	\begin{equation*}
		S\bj{(\hat n_k), (\tilde n_{\ell})}<S\bj{(n_k), (\tilde n_{\ell})}.
	\end{equation*}
\end{lemma}

\begin{proof}
	Consider the following identity:
	\begin{equation}\label{2}
		S\bj{(n_k), (\tilde n_{\ell})}
		=\sum_{k=1}^{k'}\frac{1}{x^{n_k}y^k}\sum_{\ell=1}^{\infty}\frac{z-k\ell}{x^{\ell}y^{\tilde n_{\ell}}}
		+\sum_{k=k'+1}^{\infty}\frac{1}{x^{n_k}y^k}\sum_{\ell=1}^{\infty}\frac{z-k\ell}{x^{\ell}y^{\tilde n_{\ell}}}.
	\end{equation}
	We claim that the last double sum is negative.
	For otherwise we would have
	\begin{equation*}
		0\le\sum_{k=k'+1}^{\infty}\frac{1}{x^{n_k}y^k}\sum_{\ell=1}^{\infty}\frac{z-k\ell}{x^{\ell}y^{\tilde n_{\ell}}}
		<\sum_{k=k'+1}^{\infty}\frac{1}{x^{n_k}y^k}\sum_{\ell=1}^{\infty}\frac{z-k'\ell}{x^{\ell}y^{\tilde n_{\ell}}},
	\end{equation*}
	implying
	\begin{equation*}
		\sum_{\ell=1}^{\infty}\frac{z-k'\ell}{x^{\ell}y^{\tilde n_{\ell}}}>0
	\end{equation*}
	and, more generally,
	\begin{equation*}
		\sum_{\ell=1}^{\infty}\frac{z-k\ell}{x^{\ell}y^{\tilde n_{\ell}}}>0
		\qtq{for}k=1,\ldots, k'.
	\end{equation*}
	Therefore the first double sum in \eqref{2}, and hence $S\bj{(n_k), (\tilde n_{\ell})}$ would be positive, contradicting \eqref{1}.
	
	Since the last double sum in \eqref{2} is negative and $x^{n_{k'+1}-n_{k'}}>1$ by our assumption $n_{k'+1}>n_{k'}$, we infer from \eqref{2} that
	\begin{align*}
		S\bj{(n_k), (\tilde n_{\ell})}
		&=\sum_{k=1}^{k'}\frac{1}{x^{n_k}y^k}\sum_{\ell=1}^{\infty}\frac{z-k\ell}{x^{\ell}y^{\tilde n_{\ell}}}
		+\sum_{k=k'+1}^{\infty}\frac{1}{x^{n_k}y^k}\sum_{\ell=1}^{\infty}\frac{z-k\ell}{x^{\ell}y^{\tilde n_{\ell}}}\\
		&>\sum_{k=1}^{k'}\frac{1}{x^{n_k}y^k}\sum_{\ell=1}^{\infty}\frac{z-k\ell}{x^{\ell}y^{\tilde n_{\ell}}}
		+\bj{x^{n_{k'+1}-n_{k'}}}\sum_{k=k'+1}^{\infty}\frac{1}{x^{n_k}y^k}\sum_{\ell=1}^{\infty}\frac{z-k\ell}{x^{\ell}y^{\tilde n_{\ell}}}\\
		&=\sum_{k=1}^{\infty}\frac{1}{x^{\hat n_k}y^k}\sum_{\ell=1}^{\infty}\frac{z-k\ell}{x^{\ell}y^{\tilde n_{\ell}}}\\
		&=S\bj{(\hat n_k), (\tilde n_{\ell})}.\qedhere
	\end{align*}
\end{proof}

\begin{proof}[Proof of Proposition \ref{p1}]
	In view of Lemma \ref{l1} we may assume that at least one of the sequences $(n_k)$ and $(\tilde n_{\ell})$ is non-constant.
	We have to prove that $S\bj{(n_k), (\tilde n_{\ell})}>0$.
	
	Assume on the contrary that $S\bj{(n_k), (\tilde n_{\ell})}\le 0$.
	We may assume by symmetry that the sequence $(n_k)$  is non-constant.
	Starting with the first index $k'$ such that $n_{k'+1}>n_{k'}$, by a finite or infinite number of successive applications of Lemma \ref{l2} we obtain that
	\begin{equation*}
		S\bj{(n_1), (\tilde n_{\ell})}<S\bj{(n_k), (\tilde n_{\ell})}\le 0,
	\end{equation*}
	where $(n_1)$ denotes the constant sequence $n_1, n_1, \ldots .$
	
	Since $S\bj{(n_1), (\tilde n_{\ell})}<0$, we may repeat the preceding reasoning by exchanging the role of the sequences $(n_k)$ and $(\tilde n_{\ell})$ to obtain that
	\begin{equation*}
		S\bj{(n_1), (\tilde n_1)}\le S\bj{(n_1), (\tilde n_{\ell})}<0.
	\end{equation*}
	(Here the first inequality is strict unless $(\tilde n_{\ell})$ is a constant sequence.)
	
	We have thus obtained the inequality $S\bj{(n_1), (\tilde n_1)}<0$, contradicting Lemma \ref{l1}.
\end{proof}

\begin{remark}
	The lemma  and its proof easily adapt to generalize the inequality to any finite number of real numbers $x_i>1$ and corresponding non-decreasing sequences $(n^i_k)$ of real numbers.
\end{remark}

\vspace{0.5em}
\section {Proof of Theorem \ref{t:main4} (v)}
\label{sec:4}
The main focus of this section is to prove $\dim_H\mathcal{U}_2=2$ and $\dim_{\operatorname{H}}\overline{\mathcal{U}_2} \setminus \mathcal{U}_2 \geq 1 $.
We start by recalling that the set $\mathcal{B}'$ is defined in \eqref{eq:AB}.
$$\mathcal{B}':=\set{(\mu, \alpha)\in\left(\Sigma_2^0\times\Sigma_2^1\right)\setminus \mathcal{C}':  \mu \preccurlyeq \sigma^m(\mu) \text{ and } \sigma^n(\alpha)\preccurlyeq\alpha \text{ for all } m,n\in \mathbb{N}_0}.$$
Subsequently, we will construct a subset of $\mathcal{U}'_2$.
For each $k\geq 2$ and $N\ge 2$, define
\[
{\mathcal{D}'}^N := \left\{
\begin{array} {lll}
	\alpha^N=(\alpha_i) : &\alpha_1\ldots \alpha_{2N} = 1^{2N-1}0 \text{ and } 0^N\prec \alpha_{kN+1}\ldots \alpha_{(k+1)\dot N}\prec 1^N \\
\end{array}
\right\}
\]
and define
\[
\overline{{\mathcal{D}'}^N} := \set{\mu^N=(\mu_i) : \mu_1\ldots \mu_{2N} = 0^{2N-1}1\text{ and } 0^N\prec \mu_{kN+1}\ldots \mu_{(k+1)\cdot N}\prec 1^N}.
\]
Next, we define
$${\mathcal{U}'_2}^N:=\overline{{\mathcal{D}'}^N}\times{\mathcal{D}'}^N\text{ and }\mathcal{U}_2^N:=\Phi^{-1}({\mathcal{U}'_2}^N).$$

In order to clarify the properties of ${\mathcal{U}'_2}^N$, we need the lemma below.
\begin{lemma}\label{l:42}
	Let $(\mu, \alpha)\in\mathcal{B}'$.
	\begin{enumerate}[\upshape (i)]
		\item If $\mu=\overline{\alpha}$ and  $\alpha\nearrow 1^\infty$, then $q_0(\mu, \alpha)= q_1(\mu, \alpha)\nearrow  2$.
		\item Let $\alpha$ be fixed. If $\mu\rightarrow \overline{\alpha}$ , then  $q_0(\mu, \alpha)\rightarrow  q_1(\mu, \alpha)$.
		\item If  $\alpha\nearrow 1^\infty$ and $\mu\rightarrow \overline{\alpha}$ then  $q_0(\mu, \alpha)\rightarrow  2 $ and $q_1(\mu, \alpha)\rightarrow 2$.
	\end{enumerate}
\end{lemma}
\begin{proof}
	(i) If $\mu=\overline{\alpha}$, we know from 	
	Proposition \ref{p:lexi of 1} that $(q_1, q_1)$ (or $(q_0, q_0)$) is the  solution of  $f_{\alpha}(q_0, q_1) =0$ and $\tilde{f}_{\overline{\alpha}}(q_0,q_1)=0$ .  Since the solution is unique by  Theorem \ref{t:main1}, we have  $q_0(\overline{\alpha}, \alpha)=q_1(\overline{\alpha}, \alpha)$.
	In particular, we know that $q_0(0^\infty, 1^\infty)=q_1(0^\infty, 1^\infty)=2$ when $\alpha\rightarrow 1^\infty$. This together with  the continuity of the map $\Phi^{-1}$ implies the result.
	\medskip
	
	(ii)  Using (i) and the fact that the map $\Phi^{-1}$  is continuous shown in Lemma \ref{p:inverse-map}, we may conclude our result.
	\medskip
	
	(iii)  It follows from (i) that  $q_0(\overline{\alpha},\alpha)=q_1(\overline{\alpha},\alpha)\nearrow 2$ when $\alpha\nearrow1^\infty$. This together with (ii) yields  $q_0(\mu, \alpha)\rightarrow  2 $ and $q_1(\mu, \alpha)\rightarrow 2$ when  $\alpha\nearrow 1^\infty$ and $\mu\rightarrow \overline{\alpha}$.
\end{proof}

The following lemma shows that ${\mathcal{U}'_2}^N\subseteq \mathcal{U}'_2$.

\begin{lemma}\label{l:45}
	If $(\mu^{N},\alpha^N)\in {\mathcal{U}'_2}^N$, then
	\begin{enumerate}[\upshape (i)]
		\item $\mu^{N}\prec\sigma^m(\mu^{N})\prec\alpha^N$, $\mu^{N}\prec\sigma^n(\alpha^N)\prec\alpha^N$ for all $m, n \in\mathbb{N}$.
		\item $\lim_{N\rightarrow \infty} \alpha^{N}=1^\infty$,
		$\lim_{N\rightarrow \infty} \mu^{N}=\overline{\alpha^N}$.
		\item  $\lim_{N\rightarrow \infty}q_0(\mu^{N}, \alpha^{N} )=2,\;\; \lim_{N\rightarrow \infty}q_1(\mu^{N}, \alpha^{N} ) = 2$.
		\item  If $q_1(\mu^{N},\alpha^N)\le q_0(\mu^{N},\alpha^N)$, then $q_1(\mu^{N},\alpha^N)<2$. If $q_0(\mu^{N},\alpha^N)\le q_1(\mu^{N},\alpha^N)$, then $q_0(\mu^{N},\alpha^N)<2$.
	\end{enumerate}
\end{lemma}
\begin{proof}
	(i) We know from the definition of ${\mathcal{U}'_2}^N $, there are at most $2N-2$ consecutive zero digits (or one digits) appearing in $\alpha_{2N+i}\alpha_{2N+i+1}\cdots$ for some $i\ge 0$. Symmetrically, there are at most $2N-2$ consecutive zero digits (or one digits) appearing in $\mu_{2N+i}\mu_{2N+i+1}\cdots$ for some $i\ge 0$.
	For all $m, n\in \mathbb{N}$ and $N \geq 2$, we have
	$$\mu^{N}\prec 0^{2N-2} 10^\infty\prec \sigma^n(\mu^{N})\prec 1^{2N-2} 01^\infty \prec\alpha^N$$
	and
	$$\mu^{N}\prec 0^{2N-2} 10^\infty\prec\sigma^m(\alpha^{N})\prec 1^{2N-2} 01^\infty \prec \alpha^N.$$
	\medskip
	
	(ii) follows from the definitions of $\mu^{N}$ and $\alpha^N$.
	\medskip
	
	(iii) It follows from (ii) and Lemma \ref{l:42} that $q_0(\mu^{N},\alpha^N)\rightarrow 2$  and $q_1(\mu^{N},\alpha^N)\rightarrow 2$ as $N\rightarrow \infty$.
	\medskip
	
	(iv) If $q_1(\mu^{N},\alpha^N)\le q_0(\mu^{N},\alpha^N)$, then it follows that $1\le\pi_{q_0, q_1}(\alpha^N)<1/(q_1(\mu^{N},\alpha^N)-1)$ and $q_1(\mu^{N},\alpha^N)<2$.
	The other case is shown similarly by $1\le\tilde{\pi}_{q_0, q_1}(\mu^N)<1/(q_0(\mu^{N},\alpha^N)-1).$
\end{proof}
The following lemma \ref{l:continu} will play an important role in the proof of $\dim_H\mathcal{U}_2^N=2$.

\begin{lemma}\label{l:continu}
	Assume $(q_0, q_1), (q'_0, q'_1)\in \mathcal{U}_2^N$ with $q_0\neq q'_0$ and $q_1\neq q'_1$, and $N$ is sufficiently large. Let $C=(2+\varepsilon(N))^{-2N+3}$, where $\varepsilon(N)$ is an arbitrarily small number depending on $N$. If $$d_{\mathbb{R}^2}((q_0, q_1), (q'_0, q'_1))\le C (2+\varepsilon(N))^{-m}$$ for some positive integer $m$, then $\alpha_i(q_0, q_1)=\alpha_i(q'_0, q'_1)$ and $\mu_i(q_0, q_1)=\mu_i(q'_0, q'_1)$ for all $i=1,\cdots, m$.
\end{lemma}
This lemma will be proved by the following two lemmas. We recall from Section \ref{sec:3} that $\tilde{\pi}_{q_0,q_1}(\mu)=q_1/q_0$.
\begin{lemma}\label{l:47} For sufficiently large  $N$, we have
	\begin{enumerate}[\upshape (i)]
		\item  If  $q'_1<q_1$, then
		$$\pi_{q'_0, q'_1}(1^{2N-2}0^\infty)-\pi_{q_0, q_1}(1^\infty)>0.$$
		\item If $q_0<q'_0$, then
		$$\tilde{\pi}_{q_0, q_1}(0^{2N-2}1^\infty)-\tilde{\pi}_{q'_0, q'_1}(0^\infty)>0.$$
		\item If $q_0<q'_0$ and $q_1<q'_1$, then
		$$\frac{\pi_{q_0, q_1}(1^{2N-1}0^\infty)}{q_0}-\frac{\pi_{q'_0, q'_1}(1^\infty)}{q'_0}>0\qtq{and}\frac{\tilde{\pi}_{q_0, q_1}(0^{2N-1}1^\infty)}{q_1}-\frac{\tilde{\pi}_{q'_0, q'_1}(0^\infty)}{q'_1}>0.$$
	\end{enumerate}
\end{lemma}
\begin{proof}
	(i) We have
	\begin{equation}\label{e:62}
		\pi_{q'_0, q'_1}(1^{2N-2}0^\infty)-\pi_{q_0, q_1}(1^{2N-2}1^\infty)=\sum_{i=1}^{2N-2}(\frac{1}{{q'}^i_1}-\frac{1}{q^i_1})-\frac{1}{{q}_1^{2N-2}({q}_1-1)}>0
	\end{equation}
	The last inequality in \eqref{e:62} is true for sufficiently large $N$, since $q_1<q'_1$ means $\sum_{i=1}^{2N-2}({q'}^{-i}_1-q^{-i}_1)$ is positive and increasing in $N$, ${q}_1^{2-2N}/({q}_1-1)$ is decreasing in $N$ and $\lim _{N \rightarrow \infty}{q}_1^{2-2N}/({q}_1-1)=0$.
	
	\medskip
	(ii) Similarly, we have
	\begin{equation*}
		\begin{split}
			&\tilde{\pi}_{q_0, q_1}(0^{2N-2}1^\infty)-\tilde{\pi}_{q'_0, q'_1}(0^\infty)=\sum_{i=1}^{2N-2}(\frac{1}{q^i_0}-\frac{1}{{q'}^i_0})-\frac{1}{{q}_0^{2N-2}({q}_0-1)}>0\\
		\end{split}
	\end{equation*}
	
	\medskip
	(iii) By noting ${q_0}^i{q_1}^j<{q'_0}^i{q'_1}^j$ and  a similar argument as (i) , we get
	\begin{equation*}
		\begin{split}
			&\frac{\pi_{q_0, q_1}(1^{2N-1}0^\infty)}{q_0}-\frac{\pi_{q'_0, q'_1}(1^\infty)}{q'_0}
			=\sum_{i=1}^{2N-1}(\frac{1}{q_0q^i_1}-\frac{1}{q'_0{q'}^i_1})-\frac{1}{q'_0{q'}_1^{2N-1}({q'}_1-1)}>0,
		\end{split}
	\end{equation*}
	\begin{equation*}
		\begin{split}
			&\frac{\tilde{\pi}_{q_0, q_1}(0^{2N-1}1^\infty)}{q_1}-\frac{\tilde{\pi}_{q'_0, q'_1}(0^\infty)}{q'_1}
			=\sum_{i=1}^{2N-1}(\frac{1}{q_1q^i_0}-\frac{1}{{q'}_1{q'}^i_0})-\frac{1}{q'_1{q'}_0^{2N-1}({q'}_0-1)}>0.
		\end{split}
	\end{equation*}
\end{proof}

\begin{lemma}\label{l:Dist}
	Assume $(q_0, q_1), \; (q'_0, q'_1)\in \mathcal{U}_2^N$ and $N$ is sufficiently large.  Let  $m$ be the smallest index that $\mu_i(q_0, q_1)\neq\mu_i(q'_0, q'_1)$ and $n$ be the smallest index that $\alpha_i(q_0, q_1)\neq\alpha_i(q'_0, q'_1)$.
	\begin{enumerate}[\upshape (i)]
		\item If $q'_0>q_0$ and $q_1>q'_1$, then
		\begin{equation*}
			q'_0-q_0> \left\{
			\begin{array}{clrr}
				&(2+\varepsilon(N))(2+\varepsilon(N))^{-n} & \text{ if } \alpha(q_0, q_1)\prec\alpha(q'_0, q'_1),\\
				&(2+\varepsilon(N))^{-2N+3}(2+\varepsilon(N))^{-n}  & \text{ if } \alpha(q_0, q_1)\succ\alpha(q'_0, q'_1),\\
			\end{array}
			\right.
		\end{equation*}
		\begin{equation*}
			q_1-q'_1> \left\{
			\begin{array}{clrr}
				&(2+\varepsilon(N))(2+\varepsilon(N))^{-m} & \text{ if } \mu(q_0, q_1)\prec\mu(q'_0, q'_1),\\
				&(2+\varepsilon(N))^{-2N+3}(2+\varepsilon(N))^{-m}  & \text{ if } \mu(q_0, q_1)\succ\mu(q'_0, q'_1).\\
			\end{array}
			\right.
		\end{equation*}
		\item  If $q'_0>q_0$ and $q'_1>q_1$, then
		$$q'_1-q_1>(2+\varepsilon(N))^{-2N+3}(2+\varepsilon(N))^{-m}  \qtq{and}q'_0-q_0>(2+\varepsilon(N))^{-2N+3}(2+\varepsilon(N))^{-n} .$$
	\end{enumerate}
	Here $\varepsilon(N)\rightarrow 0\text{ as }N\rightarrow \infty$.
\end{lemma}
\begin{proof}
	For simplicity, we write $\mu=(\mu_i):=\mu(q_0, q_1)$, $\mu'=(\mu'_i):=\mu(q'_0, q'_1)$, $\alpha=(\alpha_i):=\alpha(q_0, q_1)$ and  $\alpha'=(\alpha'_i):=\alpha(q'_0, q'_1)$. It is clear that $(\mu, \alpha), (\mu', \alpha')\in {\mathcal{U}'_2}^N$.
	
	First we have  by  (iii) and (iv) of Lemma \ref{l:45} that there exists a positive number $\varepsilon(N)$ such that
	$q_0, q'_0, q_1, q'_1$ all belong to $(1, 2+\varepsilon(N)]$ and $\varepsilon(N)\rightarrow 0 \text{ as } N\rightarrow \infty.$
	
	(ia) If  $\alpha\prec\alpha'$ with $\alpha_n<\alpha'_n$,  we have
	\begin{equation*}
		\begin{split}
			q'_0-q_0&=q'_1\pi_{q'_0, q'_1}((\alpha'_i))-q_1\pi_{q_0, q_1}((\alpha_i))\\
			&=\pi_{q'_0, q'_1}(1^{2N-2}0\alpha_{2N+1}\cdots\alpha_{n-1}1\alpha'_{n+1}\cdots)-\pi_{q_0, q_1}(1^{2N-2}0\alpha_{2N+1}\cdots\alpha_{n-1}0\alpha_{n+1}\cdots)\\
			&>\pi_{q'_0, q'_1}(1^{2N-2}0^{n-2N}10^\infty)-\pi_{q_0, q_1}(1^\infty)\\
			&=\pi_{q'_0, q'_1}(1^{2N-2}0^\infty)-\pi_{q_0, q_1}(1^\infty)+\frac{1}{{q'}^{2N-1}_1{q'}^{n-2N}_0}>\frac{1}{{q'}^{2N-1}_1{q'}^{n-2N}_0}.\\
		\end{split}
	\end{equation*}	
	The last  inequality follows from (i) of Lemma \ref{l:47}.
	Hence,
	$$q'_0-q_0>\frac{1}{{q'}^{2N-1}_1{q'}^{n-2N}_0}\geq (2+\varepsilon(N))(2+\varepsilon(N))^{-n}.$$
	\medskip
	
	(ib) If $\alpha\succ\alpha'$ with  $\alpha_n>\alpha'_n$. Similary, we obtain
	\begin{equation*}
		\begin{split}
			q'_0-q_0&=q'_1\pi_{q'_0, q'_1}((\alpha'_i))-q_1\pi_{q_0, q_1}((\alpha_i))\\
			&=\pi_{q'_0, q'_1}(1^{2N-2}0\alpha_{2N+1}\cdots\alpha_{n-1}0\alpha'_{n+1}\cdots)-\pi_{q_0, q_1}(1^{2N-2}0\alpha_{2N+1}\cdots\alpha_{n-1}1\alpha_{n+1}\cdots)\\
			&>\pi_{q'_0, q'_1}(1^{2N-2}0\alpha_{2N+1}\cdots\alpha_{n-1}0^{N-1}0^{N-1}10^\infty)-\pi_{q_0, q_1}(1^\infty)\\
			&>\pi_{q'_0, q'_1}(1^{2N-2}0^\infty)-\pi_{q_0, q_1}(1^\infty)+\frac{1}{{q'}_1^{2N-1}{q'}_0^{2N-1}{{q'}}_{\alpha_{2N+1}}\cdots {q'}_{\alpha_{n-1}}}\\
			&>\frac{1}{{q'}_1^{2N-1}{q'}_0^{2N-1}{{q'}}_{\alpha_{2N+1}}\cdots {q'}_{\alpha_{n-1}}}\geq (2+\varepsilon(N))^{-2N+3}(2+\varepsilon(N))^{-n}.\\
		\end{split}
	\end{equation*}
	
	\medskip
	(ic) If $\mu\prec\mu'$ with $\mu_m<\mu'_m$. We have
	\begin{equation*}
		\begin{split}
			q_1-q'_1&=q_0\tilde{\pi}_{q_0, q_1}((\mu_i))-q'_0\tilde{\pi}_{q'_0, q'_1}((\mu'_i))\\
			&=\tilde{\pi}_{q_0, q_1}(0^{2N-2}1\mu_{2N+1}\cdots\mu_{m-1}0\mu_{m+1}\cdots)-\tilde{\pi}_{q'_0, q'_1}(0^{2N-2}1\mu_{2N+1}\cdots\mu_{m-1}1\mu'_{m+1}\cdots)\\
			&>\tilde{\pi}_{q_0, q_1}(0^{2N-2}1^{m-2N}01^\infty)-\tilde{\pi}_{q'_0, q'_1}(0^\infty)\\
			&=\tilde{\pi}_{q_0, q_1}(0^{2N-2}1^\infty)-\tilde{\pi}_{q'_0, q'_1}(0^\infty)+\frac{1}{q_0^{2N-1}q_1^{m-2N}}>\frac{1}{q_0^{2N-1}q_1^{m-2N}}.\\
		\end{split}
	\end{equation*}
	The last inequality follows from (ii) of Lemma \ref{l:47}.
	Hence,
	$$q'_1-q_1>\frac{1}{q_0^{2N-1}q_1^{m-2N}}\geq (2+\varepsilon(N))(2+\varepsilon(N))^{-m} .$$
	
	\medskip
	(id) If $\mu\succ\mu'$ with $\mu_m>\mu'_m$. Similarly, we have
	\begin{equation*}
		\begin{split}
			q_1-q'_1&=q_0\tilde{\pi}_{q_0, q_1}((\mu_i))-q'_0\tilde{\pi}_{q'_0, q'_1}((\mu'_i))\\
			&=\tilde{\pi}_{q_0, q_1}(0^{2N-2}1\mu_{2N+1}\cdots\mu_{m-1}1\mu_{m+1}\cdots)-\tilde{\pi}_{q'_0, q'_1}(0^{2N-2}1\mu_{2N+1}\cdots\mu_{m-1}0\mu'_{m+1}\cdots)\\
			&>\tilde{\pi}_{q_0, q_1}(0^{2N-2}1\mu_{2N+1}\cdots\mu_{m-1}1^{N-1}1^{N-1}01^\infty)-\tilde{\pi}_{q'_0, q'_1}(0^\infty)\\
			&>\tilde{\pi}_{q_0, q_1}(0^{2N-2}1^\infty)-\tilde{\pi}_{q'_0, q'_1}(0^\infty)+\frac{1}{q_1^{2N-1}{q}^{2N-1}_0{q}_{\mu_{2N+1}}\cdots q_{\mu_{m-1}}}\\
			&>\frac{1}{q_1^{2N-1}{q}^{2N-1}_0{q}_{\mu_{2N+1}}\cdots q_{\mu_{m-1}}}\geq (2+\varepsilon(N))^{-2N+3}(2+\varepsilon(N))^{-m}.\\	
		\end{split}
	\end{equation*}
	
	(ii) It remains to prove the case $q'_0>q_0$ and $q'_1>q_1$.
	By (iii) of Remark \ref{r:monotonic} \color{black}, we have $\alpha\prec\alpha'$ and $\mu\succ\mu'$. Hence
	\begin{equation*}
		\begin{split}
			\frac{1}{q_1}-\frac{1}{q'_1}&=\frac{\pi_{q_0, q_1}((\alpha_i))}{q_0}-\frac{\pi_{q'_0, q'_1}((\alpha'_i))}{q'_0}\\
			&>\frac{\pi_{q_0, q_1}(1^{2N-1}0\alpha_{2N+1}\cdots\alpha_{n-1}0^{N-1}0^{N-1}10^{\infty})}{q_0}-\frac{\pi_{q'_0, q'_1}(1^{\infty})}{q'_0}\\
			&>\frac{\pi_{q_0, q_1}(1^{2N-1}0^{\infty})}{q_0}-\frac{\pi_{q'_0, q'_1}(1^{\infty})}{q'_0}+\frac{1}{q_0^{2N}q_1^{2N}{q}_{\alpha_{2N+1}}\cdots{q}_{\alpha_{n-1}}}\\
			&>\frac{1}{q_0^{2N}q_1^{2N}{q}_{\alpha_{2N+1}}\cdots{q}_{\alpha_{n-1}}}.\\
		\end{split}
	\end{equation*}
	
	The last inequality follows from (iii) of Lemma \ref{l:47}.
	Hence,
	$$q'_1-q_1>\frac{q'_1}{q_0^{2N}q_1^{2N-1}{q}_{\alpha_{2N+1}}\cdots{q}_{\alpha_{n-1}}}>(2+\varepsilon(N))^{-2N+3}(2+\varepsilon(N))^{-n}.$$
	
	Similarly, using (iii) of Lemma \ref{l:47}, we obtain
	\begin{equation*}
		\begin{split}
			\frac{1}{q_0}-\frac{1}{q'_0}&=\frac{\tilde{\pi}_{q_0, q_1}((\mu_i))}{q_1}-\frac{\tilde{\pi}_{q'_0, q'_1}((\mu'_i))}{q'_1}\\
			&>\frac{\tilde{\pi}_{q_0, q_1}(0^{2N-1}1\mu_{2N+1}\cdots\mu_{m-1}1^{N-1}1^{N-1}01^\infty)}{q_1}-\frac{\tilde{\pi}_{q'_0, q'_1}(0^\infty)}{q'_1}\\
			&>\frac{\tilde{\pi}_{q_0, q_1}(0^{2N-1}1^\infty)}{q_1}-\frac{\tilde{\pi}_{q'_0, q'_1}(0^\infty)}{q'_1}+\frac{1}{q_0^{2N}q_1^{2N}{q}_{\mu_{2N+1}}\cdots{q}_{\mu_{m-1}}}\\
			&>\frac{1}{q_0^{2N}q_1^{2N}{q}_{\mu_{2N+1}}\cdots{q}_{\mu_{m-1}}}.\\
		\end{split}
	\end{equation*}
	Then we have
	$$q'_0-q_0>(2+\varepsilon(N))^{-2N+3}(2+\varepsilon(N))^{-m}.$$	
	\color{black}
\end{proof}
Denote by $B_n({\mathcal{U}'_2}^N)$ the set of different blocks of $n$ letters  appearing in the sequence of ${\mathcal{U}'}_2^N$, by $B_{i, j}({\mathcal{U}'_2}^N)$ the set of words $c_i\cdots c_j$ of the sequence $(c_i)$ of ${\mathcal{U}'_2}^N$, and by $|A|$ the cardinality of the set $A$.
We know from the definition of ${\mathcal{U}'_2}^N$ that
\begin{equation}\label{e:number}
	\begin{split}
		&|B_{nN}({\mathcal{U}'_2}^N)|=((2^{N}-2)^{n-2})^2 \text{ for all } n\ge 2,\\
		&|B_{kN+1,\; nN}({\mathcal{U}'_2}^N)|=((2^{N}-2)^{n-k} )^2 \text{ for all } 2\le k\le n.
	\end{split}
\end{equation}

The proof of Theorem \ref{t:hausdorff dimension} is analogous to that of \cite[Theorem 1.6 (ii) ]{KomKonLi2015}.
\begin{theorem}\label{t:hausdorff dimension}
	$\dim_H \mathcal{U}_2=2$ and  $\dim_{\operatorname{H}}\overline{\mathcal{U}_2} \setminus \mathcal{U}_2 \ge 1 $.
\end{theorem}

\begin{proof}
	Since  $\mathcal{U}_{2}^N\subseteq \mathcal{U}_2 $ , it suffices to prove $\dim_H \mathcal{U}_2^N=2$.
	By Lemma \ref{l:continu}, consider a finite cover $\cup I_j$ of $\mathcal{U}_2^N$ by squares $I_j$ of diameter $|I_j|\leq C (2+\varepsilon(N))^{-N}$. For each positive integer $k$, we denote $J_k$ as  the set of indices $j$ that satisfy:
	\begin{equation}\label{e:hausdorff dimension 1}
		C(2+\varepsilon(N))^{-(k+1)N}<|I_j|\le C(2+\varepsilon(N))^{-kN}.
	\end{equation}
	It is clear that $J_k=\emptyset$ for all $k\geq n$,
	when $n$ is sufficiently large and  satisfies $C (2+\varepsilon(N))^{-nN}<|I_j|$ for all $j$.
	
	If $j\in J_k$ and $(q_0,q_1),(q_0',q_1')\in \mathcal{U}_2^N\cap I_j$, we can deduce from Lemma \ref{l:continu} that the first $kN$ letters of $\alpha(q_0, q_1)$ and $\alpha(q'_0, q'_1)$ are the same, so are $\mu(q_0, q_1)$ and $\mu(q'_0, q'_1)$.
	Therefore,  at most $|B_{kN+1,nN}({\mathcal{U}'_2}^N)|$ elements of $B_{nN}({\mathcal{U}'_2}^N)$ may occur for the bases $(q_0,q_1)\in \mathcal{U}_2^N\cap I_j$. Hence  $$\left|B_{n N}({\mathcal{U}'_2}^N)\right|\leq   \sum_k \sum_{j\in J_k} \left|B_{k N+1, n N}\left({\mathcal{U}'_2}^N\right)\right|.$$
	By using \eqref{e:number} , we have
	$$((2^{N}-2)^{n-2})^2\leq \sum_k \sum_{j\in J_k} ((2^{N}-2)^{n-k} )^2.$$
	It may be simplified to
	$$ \left(2^{N}-2\right)^{-4} \leq \sum_k \sum_{j\in J_k} \left(2^{N}-2\right)^{-2k}.$$
	Set $(2^{N}-2)^2=2^{\tau N}$.  $\tau:=\tau(N)\in(1,2)$  and  $\tau(N)\rightarrow 2$ when $N\rightarrow \infty.$
	The preceding inequality  may be rewritten  in the form
	\begin{equation}\label{e:hausdorff dimension 2}
		2^{-3\tau N} \leq \sum_{k}\sum_{j\in J_k} 2^{-(k+1)\tau N}.
	\end{equation}	
	The left inequality of \eqref{e:hausdorff dimension 1} means $(2+\varepsilon(N))^{-(k+1)N}<C^{-1}|I_j|$, which is equivalent to
	$$(2^{-(k+1)N})^{\log_2{(2+\varepsilon(N))}}<C^{-1}|I_j|$$
	and then we obtain $$2^{-(k+1)\tau N}<(C^{-1}|I_j|)^{\tau\log_{2+\varepsilon(N)}2}.$$	
	Applying \eqref{e:hausdorff dimension 2}, we have
	$$
	2^{-3\tau N} < \sum_{k}\sum_{j\in J_k}(C^{-1}|I_j|)^{\tau\log_{2+\varepsilon(N)}{2}},
	$$
	or equivalently
	$$\sum_{j}|I_j|^{\tau \log_{2+\varepsilon(N)}{2}} >	2^{-3\tau N}C^{\tau \log_{2+\varepsilon(N)}{2}}.$$
	Since the right side is positive and depends only on $N$, we conclude that 
	$$\dim_H(\mathcal{U}_2^N)\geq \tau \log_{2+\varepsilon(N)}{2}\qtq{with}\tau \log_{2+\varepsilon(N)}{2}\rightarrow2\qtq{as}N \rightarrow \infty.$$
	Now we need to show that $\dim_H \overline{\mathcal{U}_2}\setminus\mathcal{U}_2\ge 1$. Set $${\mathcal{E}'}^N:=\{(0^{2N-1}1)^\infty\}\times{\mathcal{D}'}^N.$$
	Thus, we have ${\mathcal{E}'}^N\subseteq \overline{\mathcal{U}'_2}\setminus\mathcal{U}'_2$.  This follows from the definition of ${\mathcal{E}'}^N$ and Lemma \ref{l:U'} (below).
	Hence $\Phi^{-1}({\mathcal{E}'}^N)\subseteq  \overline{\mathcal{U}_2}\setminus\mathcal{U}_2$. Using a similar argument as above, we may show that
	$\dim_H \Phi^{-1}({\mathcal{E}'}^N)\ge 1$, which means  $\dim_H \overline{\mathcal{U}_2}\setminus\mathcal{U}_2\ge 1$.
\end{proof}
\section{Proof of  Theorem \ref{t:main2} (i)-(vi) }
\label{sec:5}
In this section, we shall give the \color{black} main \color{black} proof of Theorem \ref{t:main2}. For this endeavor, we first recall  some results from \cite{KalKonLanLi2020,DeVKomLor2016}.
\begin{lemma} \cite[Theorem 2.1]{DeVKomLor2016}\label{l:SE} Let $(\mu_i)\in \Sigma_2^0$ and $(\alpha_i)\in \Sigma_2^1$. Assume $(x_i)\in\Sigma_2$.
	\begin{enumerate}[\upshape(i)]
		\item 	If $\sigma^j((x_i))\succ(\mu_i)$ for all $j\ge 1$. Then there exists a sequence $1<n_{1}<n_{2}<\cdots$ such that for each $i\geq1$,
		\begin{equation*}\label{e:ls_n}
			x_{n_i}=1 \text{ and }	x_{k+1}\cdots x_{n_{i}}\succ\mu_{1}\cdots\mu_{n_{i}-k}\text{ if } 1\leq k<n_{i}.
		\end{equation*}	
		\item 	If $\sigma^j((x_i))\prec(\alpha_i)$ for all $j\ge 1$. Then there exists a sequence $1<m_{1}<m_{2}<\cdots$ such that for each $i\geq1$,
		$$x_{m_i}=0 \text{ and }	x_{k+1}\cdots x_{m_{i}}\prec\alpha_{1}\cdots\alpha_{m_{i}-k}\text{ if } 1\leq k<m_{i}.$$
	\end{enumerate}
\end{lemma}
\begin{lemma}\cite[Lemma 3.2]{KalKonLanLi2020}\label{l:periodic_PR_LN}
	Assume $\mu=(\mu_1\cdots\mu_m)^\infty$ and $m$ is  minimal with $\sigma^j(\mu)\succcurlyeq\mu$ for all $j\ge 0$. Then
	$$\mu_{i+1}\cdots\mu_{m}\succ\mu_1\cdots\mu_{m-i}\text { for all } 0<i<m;$$
	Assume $\alpha=(\alpha_1\cdots\alpha_n)^\infty$ and  $n$ is  minimal  with $\sigma^j(\alpha)\preccurlyeq\alpha$   for all $j\ge 0$ Then  $$\alpha_{i+1}\cdots\alpha_{n}\prec\alpha_1\cdots\alpha_{n-i} \text { for all } 0<i<n.$$
\end{lemma}

\begin{lemma}\label{l:V0}
	Let $m, n\in \mathbb{N}$. The following statements are equivalent.
	\begin{enumerate}[\upshape (i)]
		\item  $\left.\mu=\left(\mu_1 \cdots \mu_m \alpha_1 \cdots \alpha_n\right)^{\infty}\text{, } \alpha=(\alpha_1 \cdots \alpha_n \mu_1 \cdots \mu_m\right)^{\infty},$
		\item  $\sigma^{m+n}(\mu)=\mu\text{, }  \sigma^m(\mu)=\alpha,$
		\item  $\sigma^{m+n}(\alpha)=\alpha\text{, }  \sigma^n(\alpha)=\mu,$
		\item  $\sigma^m(\mu)=\alpha\text{, }  \sigma^n(\alpha)=\mu.$
	\end{enumerate}
\end{lemma}

\begin{proof}
	$(i)\Rightarrow(ii),(i)\Rightarrow(iii)\text{ and }(i)\Rightarrow(iv)$ are obvious.
	Following we prove $(ii)\Rightarrow(i)$. $\sigma^{m+n}(\mu)=\mu, \sigma^m(\mu)=\alpha$ implies
	$$\mu=(\mu_1\cdots  \mu_m \mu_{m+1}\cdots  \mu_{m+n})^\infty\text{ and }(\mu_{m+1}\cdots  \mu_{n+m}\mu_1\cdots  \mu_n)^\infty=\alpha.$$
	Let $\alpha_1\cdots\alpha_n=\mu_{m+1}\cdots \mu_{m+n}$, we have (i).
	Similarly, we have $(iii)\Rightarrow(i)$.
	
	It remains to prove $(iv)\Rightarrow(i)$. It suffices to prove  $(iv)\Rightarrow(iii)$. It is true since $\alpha=\sigma^m(\mu), \;\;\mu=\sigma^n(\alpha)$ means
	$\alpha=\sigma^m(\mu)=\sigma^{m+n}(\alpha)$.
\end{proof}

We recall from Definition \ref{def:vb2} that
\begin{equation*}
	\mathcal{U}'_2:=\{ (\mu, \alpha)\in \Sigma_2^0\times\Sigma_2^1: \mu\prec\sigma^{n}(\mu)\prec\alpha, \mu\prec\sigma^{n}(\alpha)\prec\alpha,
	\text{ for every } n \in \mathbb{N}\},
\end{equation*}
and
\begin{equation*}
	\mathcal{V}'_2:=\{ (\mu, \alpha)\in\Sigma_2^0\times\Sigma_2^1: \mu\preccurlyeq\sigma^{n}(\mu)\preccurlyeq\alpha,  \mu\preccurlyeq
	\sigma^{n}(\alpha)\preccurlyeq\alpha,\text{ for every } n \in \mathbb{N}_0\}.
\end{equation*}
$\overline{\mathcal{U}'_2}$ denotes the closure of set $\mathcal{U}'_2$. In order to
characterize  $\overline{\mathcal{U}'_2}$, $ \mathcal{V}'_2$ is divided into nine situations in the following lemma.
\begin{lemma}\label{l:V'} Suppose $(\mu,\alpha)\in \mathcal{V}'_2$. We have the following possibilities.
	\begin{enumerate}[\upshape (i)]
		\item $\mu\prec\sigma^{i}(\mu)\prec\alpha \qtq{and}\mu\prec\sigma^{j}(\alpha)\prec\alpha,$
		\item $\mu\prec\sigma^{i}(\mu)\prec\alpha \qtq{and}\mu\preccurlyeq\sigma^{j}(\alpha)\prec\alpha,$
		\item $\mu\prec\sigma^{i}(\mu)\prec\alpha\qtq{and}\mu\prec\sigma^{j}(\alpha)\preccurlyeq\alpha,$
		\item $\mu\prec\sigma^{i}(\mu)\preccurlyeq\alpha\qtq{and}\mu\prec\sigma^{j}(\alpha)\preccurlyeq\alpha,$
		\item $\mu\prec\sigma^{i}(\mu)\preccurlyeq\alpha\qtq{and}\mu\prec\sigma^{j}(\alpha)\prec\alpha,$
		\item $\mu\preccurlyeq\sigma^{i}(\mu)\prec\alpha\qtq{and}\mu\prec\sigma^{j}(\alpha)\prec\alpha,$
		\item $\mu\preccurlyeq\sigma^{i}(\mu)\prec\alpha\qtq{and}\mu\preccurlyeq\sigma^{j}(\alpha)\prec\alpha,$
		\item $\mu\preccurlyeq\sigma^{i}(\mu)\prec\alpha\qtq{and}\mu\prec\sigma^{j}(\alpha)\preccurlyeq\alpha,$
		\item $\mu\preccurlyeq\sigma^{i}(\mu)\preccurlyeq\alpha\qtq{and}\mu\preccurlyeq\sigma^{j}(\alpha)\preccurlyeq\alpha.$
	\end{enumerate}
	\color{black} These cases hold for all \( i, j \in \mathbb{N} \), and here each weak equality means that there exists at least one \( i, j \in \mathbb{N} \) such that the equalities hold. \color{black}
\end{lemma}
\begin{proof}
	It follows from Lemma \ref{l:V0}.
\end{proof}

\begin{lemma}\label{l:V'compact}
	The set $\mathcal{V}'_2$ is compact.	
\end{lemma}
\color{black}	
\begin{proof}
	Since $\Sigma_2\times \Sigma_2$ is compact and $\mathcal{V}'_2\subseteq \Sigma_2\times \Sigma_2$, it suffices to prove $\mathcal{V}'_2$ is closed. We show
	that the complement of $\mathcal{V}'_2$ in $\Sigma_2\times \Sigma_2$ is open. Let $ (\mu, \alpha) \in (\Sigma_2 \times \Sigma_2) \setminus \mathcal{V}'_2 $ with $\mu = (\mu_i)$ and $\alpha = (\alpha_i)$. Then there exist two integers $k\text{ and } m\in\mathbb{N}$ such that at least one of the following four inequalities hold.
	\begin{align*}
		\alpha_{k+1} \cdots \alpha_{k+m} \prec \mu_1 \cdots \mu_m , \\
		\alpha_{k+1} \cdots \alpha_{k+m} \succ \alpha_1 \cdots \alpha_m,\\
		\mu_{k+1} \cdots \mu_{k+m} \prec \mu_1 \cdots \mu_m , \\
		\mu_{k+1} \cdots \mu_{k+m} \succ \alpha_1 \cdots \alpha_m .
	\end{align*}
	It is clear that if $(\mu', \alpha')$ is sufficiently close to $(\mu, \alpha)$,  then the corresponding inequality holds for $(\mu', \alpha')$, too, and therefore  $(\mu', \alpha')$ also belongs to the complement of $\mathcal{V}'_2 $.
\end{proof}
\color{black}
Since by Lemma \ref{l:V'compact}, \( \overline{\mathcal{U}'_2} \subseteq \mathcal{V}'_2 \), we now proceed to provide a lexicographic characterization of \( \overline{\mathcal{U}'_2} \) in the following lemma.
\begin{lemma}\label{l:U'}
	If  $(\mu, \alpha)$ satisfies (i)-(viii) listed in Lemma \ref{l:V'},
	then $(\mu, \alpha)\in\overline{\mathcal{U}'_2}$.
\end{lemma}

\begin{proof}
	\color{black}Case (i) is clearly. \color{black}We  prove that for each case (ii)-(viii) listed in Lemma \ref{l:V'}, there exist $(\mu^k, \alpha^\ell)\in\mathcal{U}'_2$ such that  $\mu^k\rightarrow \mu$ and $\alpha^\ell\rightarrow \alpha$ as $k\rightarrow \infty$ and $\ell\rightarrow \infty$. This implies that
	$(\mu, \alpha)\in\overline{\mathcal{U}'_2}$. We only prove  cases (ii)-(iv) and (viii). We give the construction of sequence in (v)-(vii) and omit their proof.
	Let $\mu=(\mu_i)$ and  $\alpha=(\alpha_i)$.

	{\bf{Case (ii)}.} We prove  case  (ii) listed in Lemma \ref{l:V'}
	by constructing sequences $\mu^k$ that satisfy
	$
	\mu^k\nearrow \mu \;\text{as}\; k\rightarrow \infty,
	$
	and
	\begin{equation}\label{e:unique0}
		\mu^k\prec\sigma^i (\mu^k)\prec\alpha \qtq{and}\mu^k\prec\sigma^j(\alpha)\prec\alpha\;\text{for all}\; i, j\in\mathbb{N}.
	\end{equation}
	Since \eqref{e:unique0} means  $(\mu^k, \alpha)\in \mathcal{U}'_2$ and $(\mu^k, \alpha)\rightarrow(\mu, \alpha) $ when 	$\mu^k\nearrow \mu$.
	We define $$\mu^k:=(\mu_1\mu_2\cdots\mu_{n_k})(\mu_1\mu_2\cdots\mu_{n_{k+1}})^\infty.$$
	Here $n_1,n_2,\cdots,n_k,n_{k+1}\cdots$  are given as in  Lemma \ref{l:SE} (i). Then we obtain
	\begin{equation}\label{e:u2}
		\mu_{m+1}\cdots \mu_{n_{k}}\succ\mu_{1}\cdots\mu_{n_{k}-m}\qtq{for all} 1\le m\le n_{k}-1,
	\end{equation}
	\begin{equation}\label{e:u3}
		\mu_{m+1}\cdots \mu_{n_{k+1}}\succ\mu_{1}\cdots\mu_{n_{k+1}-m}\qtq{for all} 1\le m\le n_{k+1}-1
	\end{equation}
	and
	\begin{equation}\label{e:u4}
		\mu_{1}\cdots\mu_{n_{k+1}-m}\succcurlyeq\mu_{1}\cdots\mu_{n_k}\mu_{1}\cdots\mu_{n_{k+1}-n_k-m}\qtq{for all} 1\le m\le n_{k+1}-n_k.
	\end{equation}	
	Then by \eqref{e:u3}, we have 
	\begin{equation}\label{e:u5}
		\mu_1\cdots\mu_{n_{k}}\mu_{n_{k}+1}\cdots \mu_{n_{k+1}}\succ\mu_{1}\cdots\mu_{n_{k}}\mu_{1}\cdots\mu_{n_{k+1}-n_{k}}.
	\end{equation}

	It follows from \eqref{e:u5} that $\mu^k\prec\mu^{k+1}\prec\mu$ and $\mu^k\nearrow \mu \;\text{as}\; k\rightarrow \infty$. $\mu^k\prec\mu$ together with the  condition $\mu\preccurlyeq\sigma^{j}(\alpha)\prec\alpha$ yields  $\mu^k\prec\sigma^{j}(\alpha)\prec\alpha$. \eqref{e:u2}-\eqref{e:u5}  means $\mu^k\prec\sigma^i(\mu^k)\qtq{for all}i\in\mathbb{N}$.
	
	Now it remains to prove $\sigma^i(\mu^k)\prec\alpha$. It is obvious for $i=0$. For other cases, we have
	\begin{equation}\label{e:overline U}
		\begin{split}
			\sigma^i(\mu^k)&=(\mu_{i+1}\mu_{i+2}\cdots\mu_{n_{k}})(\mu_1\mu_2\cdots\mu_{n_{k+1}})^\infty\\
			&\prec(\mu_{i+1}\mu_{i+2}\cdots\mu_{n_k})(\mu_{n_k+1}\mu_{n_k+2}\cdots
			\mu_{n_k +n_{k+1}})(\mu_{n_k +n_{k+1}+1}\mu_{n_k +n_{k+1}+2}\cdots\mu_{n_k +2n_{k+1}})\cdots\\
			&=\sigma^i(\mu)\prec\alpha\qtq{when}1\leq i\leq n_k-1,\\
			\sigma^i(\mu^k)
			&=\mu_{l+1}\mu_{l+2}\cdots\mu_{n_{k+1}}(\mu_1\mu_2\cdots\mu_{n_{k+1}})^\infty\\
			&\prec\mu_{l+1}\mu_{l+2}\cdots\mu_{n_{k+1}}(\mu_{n_{k+1}+1}\mu_{n_{k+1}+2}\cdots
			\mu_{2n_{k+1}})(\mu_{2n_{k+1}+1}\mu_{2n_{k+1}+2}\cdots\mu_{3n_{k+1}})\cdots\\
			&=\sigma^l(\mu)\prec\alpha\qtq{when} i=n_k+l \qtq{and} 0\leq l\leq n_{k+1}-1.
		\end{split}
	\end{equation}
	\medskip
	{\bf{Case (iii)}.} We prove case  (iii) listed in Lemma \ref{l:V'}
	by constructing sequence $\alpha^\ell$ satisfying $
	\alpha^\ell\nearrow \alpha \;\text{when}\; \ell\rightarrow \infty$ and
	\begin{equation*}
		\mu\prec\sigma^{i}(\mu)\prec\alpha^\ell,\;	\mu\preccurlyeq\sigma^j(\alpha^\ell)\prec\alpha^\ell \;\text{for all}\;i, j\in\mathbb{N}.
	\end{equation*}
	Clearly, $(\mu, \alpha^\ell)$ satisfies the case  (ii) listed in Lemma \ref{l:V'} for all $\ell$. By applying the result of case (ii), we have $(\mu^k, \alpha^\ell)\in \mathcal{U}'_2$ and  $(\mu^k, \alpha^\ell)\rightarrow (\mu,
	\alpha)$ as $k\rightarrow \infty$ and $\ell\rightarrow \infty$.
	
	We have, by the assumption, that $\mu\prec\sigma^{j}(\alpha)\preccurlyeq\alpha$, $\alpha=(\alpha_1\cdots\alpha_n)^\infty$.
	Set $$\alpha^\ell:=(\alpha_1\alpha_2\cdots\alpha_n)^\ell\sigma^S(\mu)\qtq{with}\sigma^S(\mu):=\sup\{\ \sigma^j(\mu): j\in\mathbb{N} \}.
	$$
	It is clear that
	$\alpha^\ell\nearrow \alpha \;\text{as}\; \ell\rightarrow \infty$ because $\sigma^i(\mu)\prec\alpha$ for all $i\in\mathbb{N}_0$.
	First we show that $\sigma^j(\mu)\prec\alpha^\ell$ for all
	$j\in\mathbb{N}_0$. We claim that
	\begin{equation}\label{e:overline U1}
		\begin{split}
			\sigma^j(\mu)\prec\alpha_1\cdots\alpha_{n}\sigma^j(\mu)\qtq{for all}j\in\mathbb{N}_0.		
		\end{split}
	\end{equation} Otherwise, we have
	$\sigma^j(\mu)\succcurlyeq\alpha_1\cdots\alpha_{n}\sigma^j(\mu)\qtq{for some}j\in\mathbb{N}_0.$ This yields that
	$$\sigma^j(\mu)\succcurlyeq\alpha_1\cdots\alpha_{n}\sigma^j(\mu)
	\succcurlyeq\alpha_1\cdots\alpha_{n}\alpha_1\cdots\alpha_{n}\sigma^j(\mu)
	\succcurlyeq\cdots \succcurlyeq(\alpha_1\cdots\alpha_{n})^{\infty}=\alpha.$$
	It is impossible by the assumption. Therefore, we have $$\mu\prec\sigma^{j}(\mu)\preccurlyeq\sigma^{S}(\mu)\prec\alpha_1\alpha_2\cdots\alpha_{n}\sigma^{S}(\mu)
	\prec\cdots\prec(\alpha_1\alpha_2\cdots\alpha_{n})^\ell\sigma^{S}(\mu)=\alpha^\ell.$$
	It is obvious that $\sigma^j(\alpha^\ell)\prec\alpha^\ell\qtq{for all}j\in\mathbb{N}$ by Lemma \ref{l:periodic_PR_LN} and (\ref{e:overline U1}). It remains to prove
	$\mu\preccurlyeq\sigma^j(\alpha^\ell)\qtq{for all}j\in\mathbb{N}_0$. Suppose that there exists a $j$ such that $\mu\succ\sigma^j(\alpha^\ell)\qtq{for some}0\leq j\leq ln-1$,
	which leads to $\mu_1\cdots\mu_{n-j}\succ\alpha_{j+1}\cdots\alpha_n$. It   implies that $\mu\succ\sigma^j(\alpha)$, which is impossible.
	Therefore, we have $\mu\preccurlyeq\sigma^j(\alpha^\ell)\qtq{for all}j\in\mathbb{N}_0$.\color{black}
	\medskip
	
	{\bf{Case (iv)}.}
	We prove the case  (iv) listed in Lemma \ref{l:V'}
	by constructing  sequences $\mu^k$ satisfying
	$
	\mu^k\nearrow \mu \;\text{as}\; k\rightarrow \infty,
	$
	and
	\begin{equation*}
		\mu^k\prec\sigma^i(\mu^k)\prec\alpha \ \qtq{and}\mu^k\prec\sigma^j(\alpha)\preccurlyeq\alpha \;\text{ for all } i, j\in\mathbb{N}.
	\end{equation*}  The sequence $\mu^k$ can be constructed the same as in case (ii).  It is clear that we have  $	\mu^k\prec\sigma^i(\mu^k)$ and $\mu^k\prec\sigma^j(\alpha)\preccurlyeq\alpha$ by the proof of case (ii). The only difference is the proof of  $ \sigma^i(\mu^k)\prec\alpha$. It  still holds , since  there are strict inequalities in \eqref{e:overline U}.
	
	By applying the result of case (iii), we can construct sequence   $\alpha^\ell$ satisfying $
	\alpha^\ell\nearrow \alpha \;\text{as}\; \ell\rightarrow \infty$ and
	\begin{equation*}
		\mu^k\prec\sigma^{i}(\mu^k)\prec\alpha^\ell,\;	\mu^k\preccurlyeq\sigma^j(\alpha^\ell)\prec\alpha^\ell \text{ for all }\; i, j\in\mathbb{N}.
	\end{equation*}
	Here  	$
	\mu^k\nearrow \mu \;\text{as}\; k\rightarrow \infty,
	$ and $
	\alpha^\ell\nearrow \alpha \;\text{as}\; \ell\rightarrow \infty$.
	
	\medskip
	
	{\bf{Cases (v-vii)}.}
	
	(v) We construct sequences $\alpha^k$ satisfying
	$\alpha^k\searrow \alpha \;\text{as}\; k\rightarrow \infty,$
	and
	\begin{equation*}\label{1.20}
		\mu\prec\sigma^i(\mu)\prec\alpha^k \text{ and } \mu \prec\sigma^j(\alpha^k)\prec\alpha^k \text{ for all } i, j\in\mathbb{N}.
	\end{equation*}
	Set $$\alpha^k:=(\alpha_1\alpha_2\cdots\alpha_{m_k})(\alpha_1\alpha_2\cdots\alpha_{m_{k+1}})^\infty.$$
	Here $m_1, m_2,\cdots, m_k\cdots$  are given as in  Lemma \ref{l:SE} (ii).
	\medskip
	
	(vi) We construct sequence $\mu^\ell\searrow \mu \;\text{as}\; \ell\rightarrow \infty,$
	and
	\begin{equation*}
		\mu^\ell\prec\sigma^i(\mu^\ell)\preccurlyeq\alpha \text{ and } \mu^\ell\prec\sigma^{j}(\alpha)\prec\alpha\;\text{for all}\; i, j\in\mathbb{N}.
	\end{equation*}
	It follows from the  the assumption that  $\mu=(\mu_1\cdots\mu_m)^\infty$.
	Set $$\mu^k:=(\mu_1\cdots\mu_m)^k\sigma^T(\alpha)\qtq{with}\sigma^T(\alpha):=\inf\{\ \sigma^j(\alpha): j\in\mathbb{N} \}.$$
	\medskip
	
	(vii) We construct the sequence $\alpha^k$ as in (v) which satisfies
	$
	\alpha^k\searrow \alpha \;\text{as}\; k\rightarrow \infty,
	$
	and
	\begin{equation*}
		\mu\preccurlyeq\sigma^i(\mu)\prec\alpha^k\ \qtq{and}\mu\prec\sigma^j(\alpha^k)\prec\alpha^k \;\text{ for all } i, j\in\mathbb{N}.
	\end{equation*}
	
	\medskip
	
	{\bf{Case (viii)}.} We show there exist
	a sequence $\alpha^\ell$ satisfying $
	\alpha^\ell\nearrow \alpha \;\text{as}\; \ell\rightarrow \infty,
	$
	and
	\begin{equation}\label{e:410}
		\mu\preccurlyeq\sigma^{i}(\mu)\prec\alpha^\ell	\text{ and }\mu\preccurlyeq\sigma^j(\alpha^\ell)\prec\alpha^\ell  \;\text{for each}\; i, j\in\mathbb{N}.
	\end{equation}
	That is case (viii) can be reduced to case (vii).
	Set $$\alpha^\ell:=(\alpha_1\cdots\alpha_n)^\ell\sigma^S((\mu_i))\qtq{with}\sigma^S((\mu_i)):=\sup\{\ \sigma^j((\mu_i)): j\in\mathbb{N} \}.
	$$
	We prove $(\mu, \alpha^\ell)$ satisfying \eqref{e:410} by using the same argument in  case (iii).
\end{proof}

Now we begin to prove the main content of Theorem \ref{t:main2}.
\begin{lemma}\label{l: V'isolated}
	If $(\mu, \alpha)\in \mathcal{V}'_2$ and satisfies
	\begin{equation}\label{e:V1}
		\sigma^{m}(\mu)=(\alpha),\;\; \sigma^{n}(\alpha)=(\mu)\text{ for some }m,n\in\mathbb{N}.
	\end{equation}
	Then $(\mu, \alpha)$ is an isolated point of  $\mathcal{V}'_2$.
\end{lemma}
\begin{proof}
	We verify that if $(\mu', \alpha')\in \mathcal{V}'_2$ and $(\mu', \alpha')$ is close enough to $(\mu, \alpha)$, then  $(\mu', \alpha')=(\mu, \alpha)$.
	By \eqref{e:V1} and Lemma \ref{l:V0},  we have $$\mu=(\mu_1\cdots\mu_m\alpha_1\cdots\alpha_n)^\infty, \alpha=(\alpha_1\cdots\alpha_n\mu_1\cdots\mu_m)^\infty.$$ Suppose $\mu'$ begins with $(\mu_1\cdots\mu_m\alpha_1\cdots\alpha_n)^kx$ and $\alpha'$ begins with $=(\alpha_1\cdots\alpha_n\mu_1\cdots\mu_m)^ky$ for some large positive integer number $k$. Here $x, y\in\{0,1\}^{m+n}$.  Then $(\mu', \alpha')\in \mathcal{V}'_2$ implies $y\preccurlyeq\alpha_1\cdots\alpha_n\mu_1\cdots\mu_m$ and $x\preccurlyeq\mu_1\cdots\mu_m\alpha_1\cdots\alpha_n$.   Similarly, we have $x\succcurlyeq \mu_1\cdots\mu_m\alpha_1\cdots\alpha_n$ and $y\succcurlyeq\alpha_1\cdots\alpha_n\mu_1\cdots\mu_m$. Hence we obtain $x=\mu_1\cdots\mu_m\alpha_1\cdots\alpha_n$ and $y=\alpha_1\cdots\alpha_n\mu_1\cdots\mu_m$.
	By iterating this process infinitely, we obtain $\mu'=\mu$ and $\alpha'=\alpha$.
\end{proof}

\begin{lemma}\label{l:dense22}
	$\mathcal{V}'_2 \setminus \overline{\mathcal{U}'_2}$ is a countable set, dense in $\mathcal{V}'_2$.
\end{lemma}
\begin{proof}
	If $(\mu,\alpha)\in\mathcal{V}'_2 \setminus \overline{\mathcal{U}'_2}$, we know from Lemmas \ref{l:U'} and \ref{l: V'isolated} that $\mu$ and $\alpha$ are periodic, which leads to $\mathcal{V}'_2 \setminus \overline{\mathcal{U}'_2}$  being a countable set.
	
	Note that $\mathcal{U}'_2$ is dense in $\overline{\mathcal{U}'_2}$. Then from the transitivity of density, to prove  $\mathcal{V}'_2 \setminus \overline{\mathcal{U}'_2}$ is dense in
	$\mathcal{V}'_2=(\mathcal{V}'_2 \setminus \overline{\mathcal{U}'_2})\cup\overline{\mathcal{U}'_2}$, \color{black} it suffices to prove that  $\mathcal{U}'_2$ belongs to the closure of $\mathcal{V}'_2 \setminus \overline{\mathcal{U}'_2}$\color{black} . 
	We  show the following for any $(\mu, \alpha)\in\mathcal{U}_2'$, there exist $(\mu^n, \alpha^m)\in \mathcal{V}'_2\setminus\overline{\mathcal{U}'_2}$ such that $(\mu^n, \alpha^m)\rightarrow (\mu, \alpha)\text{ as }n\rightarrow \infty, m\rightarrow \infty$.
	
	Set $(\mu,\alpha)=((\mu_i),(\alpha_i))\in{\mathcal{U}'_2}$. Define:
	\begin{equation*}
		\begin{split}	
			\mu^{n}:=(\mu_1\mu_2\cdots\mu_n^-\alpha_1\alpha_2\cdots\alpha_m^+)^{\infty},\alpha^{m}:=(\alpha_1\alpha_2\cdots\alpha_m^+\mu_1\mu_2\cdots\mu_n^-)^{\infty}.
		\end{split}
	\end{equation*}
	Here $n$ and $m$ satisfy the inequalities in Lemma \ref{l:SE}. 
	It is clear that  $(\mu^{n}, \alpha^{m})\rightarrow (\mu, \alpha)$ when $n\rightarrow \infty, m\rightarrow \infty$. Now we verify $(\mu^{n}, \alpha^{m})\in \mathcal{V}'_2\setminus\overline{\mathcal{U}'_2}$.
	Since $\mu^{n}$ and $\alpha^{m}$ are periodic, we only need to verify $$\mu^n\prec\sigma^i(\mu^n)\preccurlyeq \alpha^m,\;\; \mu^n\preccurlyeq\sigma^i(\alpha^m)\prec\alpha^m$$ hold for all $1\le i\le{n}+{m}-1$. We only show the first inequality, the second one can be shown similarly.

	Assume that $n>m$ without loss of generality, by Lemma \ref{l:SE} and $(\mu,\alpha)\in{\mathcal{U}'_2}$, we have
	\begin{equation}\label{e:de1}
		\begin{split}
			&\mu_1\cdots\mu_{n-i}\;\;\mu_{n-i+1}\cdots\mu_{n}^-\prec
			\mu_{i+1}\cdots\mu_{n}^-\;\;\alpha_1\cdots\alpha_{i},\;\;\;\;\;\;\;\qtq{if}1\le i\le m-1,\\
			&	\mu_1\cdots\mu_{n-i}\;\;\mu_{n-i+1}\cdots\mu_{n-i+m}\prec
			\mu_{i+1}\cdots\mu_{n}^-\;\;\alpha_1\cdots\alpha_{m}^+,\qtq{if}m\le i\le n-1,\\
			&\mu_1\cdots\mu_{m-j}\prec\alpha_{j+1}\cdots\alpha_{m}^+\qtq{for}0\leq j\leq m-1,i=n+j,\qtq{if} n\le i\le n+m-1.
		\end{split}
	\end{equation}
	\eqref{e:de1} means $\mu^{n}\prec \sigma^i(\mu^{n})$. $\sigma^i(\mu^{n})\preccurlyeq \alpha^{m}$ follows from \eqref{e:de2} below and  $\sigma^n(\mu^{n})=\alpha^{m}$.
	\begin{equation}\label{e:de2}
		\begin{split}
			&\mu_{i+1}\cdots\mu_{m+i}\prec\alpha_1\cdots\alpha_{m}^+\;\;\;\;\;\;\;\;\;\qtq{if}0\leq i\leq n-m,\\
			&\mu_{i+1}\cdots\mu_{n}^-\prec\alpha_1\cdots\alpha_{n-i}\;\;\;\;\;\;\;\;\;\qtq{if}n-m+1\leq i\leq n-1,\\
			&\alpha_{j+1}\cdots\alpha_{m}^+\mu_1\cdots\mu_j\prec\alpha_1\cdots\alpha_{m}^+\qtq{if}1\leq j\leq m-1.\\
		\end{split}
	\end{equation}
	Therefore, $\mathcal{V}'_2\setminus\overline{\mathcal{U}'_2}$ is dense in $\mathcal{V}'_2$.
\end{proof}

\begin{lemma}\label{l:U'cantor}
	$\overline{\mathcal{U}'_2}$ is a Cantor set.
\end{lemma}
\begin{proof}
	Recall that a Cantor set is a non-empty compact set having neither isolated, nor interior points. We conclude that $\overline{\mathcal{U}'_2}$ has no isolated points by the proofs of Lemmas  \ref{l:V'compact} and \ref{l:U'}. Lemma \ref{l:dense22} implies that $\overline{\mathcal{U}'_2}$ has no interior points.
\end{proof}
\color{black}
To prove $\overline{\mathcal{U}'_2} \setminus \mathcal{U}'_2$ is dense in $\overline{\mathcal{U}'_2}$, we need the following lemma.
\begin{lemma}\label{l:Uldense}
	Suppose
	$(\mu, \alpha)=((\mu_i), (\alpha_i))\in \mathcal{V}'_2$.
	\begin{enumerate}[\upshape (i)]
		\item  	If $\sigma^j(\mu)\succ \mu$ for all $j\in\mathbb{N}$, then  $\mu^{n_j}=(\mu_1\cdots\mu_{n_j})^\infty\neq 0^\infty$ satisfying
		$$\mu^{n_j}\nearrow\mu \qtq{and} \mu^{n_j}\preccurlyeq\sigma^i(\mu^{n_j})\prec\alpha \qtq{for all} i\in\mathbb{N}_0.$$ Here $n_1<n_2<\cdots$ is  given as in Lemma \ref{l:SE}.
		\item  	If $\sigma^j(\alpha)\prec\alpha$ for all $j\in\mathbb{N}$, then $\alpha^{m_k}=(\alpha_1\cdots\alpha_{m_k})^\infty\neq 1^\infty$ satisfying
		$$\alpha^{m_k}\searrow\alpha \qtq{and} \mu\prec\sigma^i(\alpha^{m_k})\preccurlyeq\alpha^{m_k} \qtq{for all} i\in\mathbb{N}_0.$$  Here $m_1<m_2<\cdots$ are  given as in Lemma \ref{l:SE}.
		\item  $\mu^{n_j}=(\mu_1\cdots\mu_{n_j})^\infty\neq 0^\infty$,	 $\alpha^{m_k}=(\alpha_1\cdots\alpha_{m_k})^\infty\neq 1^\infty$ satisfying$$ \mu^{n_j}\preccurlyeq\sigma^i(\mu^{n_j})\prec\alpha^{m_k}
		\qtq{and}  \mu^{n_j}\prec\sigma^i(\alpha^{m_k})\preccurlyeq\alpha^{m_k}\qtq{for all} i\in\mathbb{N}_0.$$
	\end{enumerate}
\end{lemma}

\begin{proof}
	(i) First, we show  $\mu^{n_j}\prec\mu^{n_{j+1}}$.  Note by Lemma  \ref{l:SE} we have  
	\begin{equation}\label{e:app_ex1}
		\mu_1\cdots \mu_{n_{j+1}-m}\prec\mu_{m+1}\cdots\mu_{n_{j+1}}\qtq{for all} 1\le m\le n_{j+1}-1
	\end{equation}
	and
	\begin{equation}\label{e:app_ex2}
		\mu_1\cdots \mu_{n_{j}-m}\prec\mu_{m+1}\cdots\mu_{n_{j}}\qtq{for all} 1\le m\le n_{j}-1.
	\end{equation}
	Let $n_{j+1}=\ell n_j+s$ with $\ell\in\mathbb{N}$ and $1\le s\le n_j$. Then  by \eqref{e:app_ex1}, we have
	\begin{equation}\label{e:app_ex3}
		\mu_1\cdots\mu_{n_j}\mu_{n_j+1}\cdots\mu_{n_{j+1}}\succ\mu_1\cdots\mu_{n_j}\mu_1\cdots \mu_{(l-1)n_j+s}\succcurlyeq (\mu_1\cdots\mu_{n_j})^{\ell}\mu_1\cdots\mu_{s}
	\end{equation}
	This implies $\mu^{n_j}\prec\mu^{n_{j+1}}$. Similarly as \eqref{e:app_ex3}, we have  $\mu^{n_{j+1}}\prec\mu$. Moreover, we have  $\mu^{n_j}\preccurlyeq\sigma^i(\mu^{n_j})$ for all $i\in\mathbb{N}_0$ by \eqref{e:app_ex2}, then it remains to prove $\sigma^i(\mu^{n_j})\prec\alpha$ for all $i\in\mathbb{N}_0$.
	\color{black}
	Assume that there exist an integer $t$ such that $\sigma^t(\mu^{n_j})\succcurlyeq\alpha$ with $1\leq t \leq n_j$, then  we have
	\begin{equation}\label{e:app_ex4}
		\alpha_1\alpha_2\cdots\alpha_{n_{j+1}-t}\preccurlyeq\mu_{t+1}\cdots\mu_{n_j}(\mu_1\cdots\mu_{n_j})^{\ell-1}\mu_1\cdots\mu_{s}
		\qtq{for}n_{j+1}=\ell n_j+s.  \end{equation}
	with $\ell\ge1$ and $1\le s\le n_j$ in the above inequality. Then by  \eqref{e:app_ex3} and  \eqref{e:app_ex4}, we have $\sigma^t(\mu)\succ\alpha$ for some  integer $t$. This  contradicts with our condition $\sigma^j(\mu)\preccurlyeq \alpha$ for all $j\in\mathbb{N}_0$.
	Therefore, $\sigma^i(\mu^{n_j})\prec\alpha$ for all $i\in\mathbb{N}_0$.	
	
	(ii) can be proved symmetrically.
	
	(iii) It follows directly from  (i) and (ii) .
\end{proof}

\begin{lemma}\label{l:dense-coutable}
	$\overline{\mathcal{U}'_2} \setminus \mathcal{U}'_2$ is dense in $\overline{\mathcal{U}'_2}$.
\end{lemma}
\begin{proof}
	We prove  $\overline{\mathcal{U}'_2} \setminus \mathcal{U}'_2$ is dense in $\overline{\mathcal{U}'_2}$ by  showing for any $(\mu, \alpha)\in
	\mathcal{U}_2'$, there exist $(\mu^k, \alpha^j)\in \overline{\mathcal{U}'_2} \setminus \mathcal{U}'_2$ and $(\mu^k, \alpha^j)\rightarrow (\mu, \alpha)$ when
	$k\rightarrow \infty, j\rightarrow \infty$. It is true by Lemma \ref{l:Uldense}.
\end{proof}

\begin{proof}[Proof of (i)-(vi) of  Theorem \ref{t:main2}]  It
	follows from Lemmas \ref{l:V'}-\ref{l:U'cantor} and Lemma \ref{l:dense-coutable}.
\end{proof}

\section{Proof of Theorem \ref{t:main2} (vii) , Theorem \ref{t:main3}   and  Theorem \ref{t:main4} (i)-(iv) }
\label{sec:6}
In this section, we will prove  Theorem \ref{t:main2} (vii) , Theorem \ref{t:main3}   and  Theorem \ref{t:main4} (i)-(iv).

Recall the definitions of  $\lambda(q_0, q_1)$, $\mu(q_0, q_1)$, $\beta(q_0, q_1)$ and $\alpha(q_0, q_1)$   given in the Introduction.
The following Lemma \ref{l:LRContinous} states that if $(q_0, q_1)$ and  $(q'_0, q'_1)$ are close enough, then either the greedy or  quasi-greedy expansions of $r_{q_0, q_1}$ and $r_{q'_0, q'_1}$ begin with  the same long length words. Similar results also hold for $\ell_{q_0, q_1}$  and $\ell_{q'_0, q'_1}$.
\begin{lemma}\label{l:LRContinous}
	\color{black}Let $(q_0, q_1), (q'_0, q'_1)\in\{(q_0, q_1)\in(1,\infty)^2:q_0+q_1\ge q_0q_1\}$.
	For any given positive integers $M$ and $N$, there exists a positive number $\epsilon$ such that if $d_{\mathbb{R}^2}((q_0, q_1), (q'_0, q'_1))<\epsilon$\color{black}, then we have
	\begin{enumerate}[\upshape (i)]
		\item
		\begin{equation*}
			\begin{split}
				&\beta_1(q'_0, q'_1)\cdots \beta_N(q'_0, q'_1)=\beta_1(q_0,q_1)\cdots \beta_N(q_0,q_1)\text{ if } \alpha(q'_0, q'_1)\succcurlyeq\alpha(q_0, q_1),\\
				&\alpha_1(q'_0, q'_1)\cdots \alpha_N(q'_0, q'_1)= \alpha_1(q_0,q_1)\cdots \alpha_N(q_0,q_1) \text{ if } \alpha(q'_0, q'_1)\prec\alpha(q_0, q_1).
			\end{split}	
		\end{equation*}	
		\item
		\begin{equation*}
			\begin{split}
				&	\lambda_1(q'_0, q'_1)\cdots \lambda_{M}(q'_0, q'_1)=\lambda_1(q_0,q_1)\cdots \lambda_{M}(q_0,q_1)\text{ if }
				\mu(q'_0, q'_1)\preccurlyeq\mu(q_0, q_1),\\
				&	\mu_1(q'_0, q'_1)\cdots \mu_{M}(q'_0, q'_1)= \mu_1(q_0,q_1)\cdots \mu_{M}(q_0,q_1) \text{ if }
				\mu(q'_0, q'_1)\succ\mu(q_0, q_1).
			\end{split}	
		\end{equation*}
	\end{enumerate}
\end{lemma}
\begin{proof}
	We will prove (i) only, (ii) may be proved  similarly.
	Write $\alpha(q_0,q_1)=(\alpha_i(q_0, q_1))$ and $\beta(q_0,q_1)=(\beta_i(q_0, q_1))$. The following inequality follows from the
	definition of the greedy algorithm.
	\begin{equation*}
		\begin{split}
			&\sum_{k=1}^n\frac{\beta_k(q_0,q_1)}{q_{\beta_1(q_0,q_1)}q_{\beta_2(q_0,q_1)}\cdots q_{\beta_k(q_0,q_1)}}>\frac{q_0}{q_1}-\frac{1}{q_{\beta_1(q_0,q_1)}q_{\beta_2(q_0,q_1)}\cdots q_{\beta_n(q_0,q_1)}}\\
			&\text{ whenever }\beta_n(q_0,q_1)=0\text{ and }n\in \set{1,2,\cdots,N}.
		\end{split}
	\end{equation*}
	The same set of inequalities holds with $(q'_0, q'_1)$ in place of $(q_0,q_1)$ if  $(q'_0, q'_1)$  is close enough to $(q_0,q_1)$. Hence
	\begin{equation}\label{e:base contin1}
		\beta_1(q'_0, q'_1)\cdots \beta_N(q'_0, q'_1)\preccurlyeq \beta_1(q_0,q_1)\cdots \beta_N(q_0,q_1).
	\end{equation}
	Consider the assumption $\alpha(q'_0, q'_1)\succcurlyeq\alpha(q_0, q_1)$.
	
	If $\alpha(q_0, q_1)=\beta(q_0, q_1)$,
	then $$\beta(q'_0, q'_1)\succcurlyeq\alpha(q'_0, q'_1)\succcurlyeq\alpha(q_0, q_1)=\beta(q_0, q_1).$$
	Hence we have
	$$	\beta_1(q'_0, q'_1)\cdots \beta_N(q'_0, q'_1)\succcurlyeq \beta_1(q_0,q_1)\cdots \beta_N(q_0,q_1).$$
	This together with \eqref{e:base contin1} proves the first inequality in (i).
	
	If $\alpha(q_0, q_1)\prec\beta(q_0, q_1)$ and $\beta_\ell(q_0, q_1)$ is the last nonzero element of  $\beta(q_0, q_1)$, then, by \eqref{eq:qgreedy2}, we have $\alpha(q_0, q_1)=(\beta_1(q_0, q_1)\cdots\beta_{\ell-1}(q_0, q_1)0)^\infty$. We claim $\beta(q'_0, q'_1)\succcurlyeq \beta(q_0, q_1)$. This  also leads to (i). 
	Otherwise,
	$$\alpha(q_0, q_1)\preccurlyeq \alpha(q'_0, q'_1)\preccurlyeq \beta(q'_0, q'_1)\prec\beta(q_0, q_1)$$ and $\beta(q_0, q_1)=\beta_1(q_0, q_1)\cdots\beta_{\ell-1}(q_0, q_1)10^\infty$
	means $$(\beta_1(q_0, q_1)\cdots\beta_{\ell-1}(q_0, q_1)0)^\infty= \alpha(q'_0, q'_1)=\beta(q'_0, q'_1),$$ it is impossible.
	
	It remains to prove the  case $\alpha(q'_0, q'_1)\prec\alpha(q_0, q_1)$. Clearly, we have
	\begin{equation}\label{e:base contin2}
		\alpha_1(q'_0, q'_1)\cdots \alpha_N(q'_0, q'_1)\preccurlyeq \alpha_1(q_0,q_1)\cdots \alpha_N(q_0,q_1).
	\end{equation}
	By the
	quasi-greedy algorithm, we have
	\begin{equation*}
		\sum_{k=1}^n\frac{\alpha_k(q_0,q_1)}{q_{\alpha_1(q_0,q_1)}q_{\alpha_2(q_0,q_1)}\cdots q_{\alpha_k(q_0,q_1)}}<\frac{q_0}{q_1},
		\;\; n\in \set{1,2,\cdots, N}.
	\end{equation*}
	The same set of inequalities holds with $(q'_0, q'_1)$ in place of $(q_0,q_1)$ if  $(q'_0, q'_1)$  is close enough to $(q_0,q_1)$, whence
	\begin{equation}\label{e:base contin3}
		\alpha_1(q'_0, q'_1)\cdots \alpha_N(q'_0, q'_1)\succcurlyeq\alpha_1(q_0,q_1)\cdots \alpha_N(q_0,q_1).
	\end{equation}
	\eqref{e:base contin2} and \eqref{e:base contin3} yield our result.
\end{proof}

Recall that
$$\mathcal{C}'=\left(\set{0^\infty, x_1\cdots x_m 1^\infty}\times \Sigma_2^1\right)\cup \left(\Sigma_2^0\times \set{1^\infty, y_1\cdots y_n 0^\infty}\right) $$ with $ x_i,y_j \in \set{0,1}$ and  $ x_m=0, y_n=1$.
And $$\mathcal{C}:=\{(q_0, q_1)\in(1,\infty)^2:q_0+q_1=q_0q_1\}.$$

\color{black}
\begin{lemma}\label{l:Ul2}
	$\mathcal{U}_2$ is unbounded. Moreover,
	$$\mathcal{U}_2=\Phi^{-1}(\mathcal{U}'_2)\cup \mathcal{C} \text{ and }  \mathcal{U}_2\setminus \mathcal{C}=\Phi^{-1}(\mathcal{U}'_2).$$
\end{lemma}
\color{black}
\begin{proof}
	We will prove $\mathcal{U}_2$ is unbounded.
	Note that $\mathcal{C}\subseteq \mathcal{U}_2$, since both $\mu({q_0,q_1})$ and $\alpha({q_0,q_1})$ have a unique double-base expansion $0^\infty$ and $1^\infty$, respectively.
	Take $q_{1,n}=1+1/n$
	and  $q_0=q_{1,n}/(q_{1,n}-1)$, we have $(q_{1,n}/(q_{1,n}-1), q_{1,n})\in\mathcal{U}_2$, $q_{1,n}\rightarrow 1$ as $n\rightarrow \infty$, and $q_0=q_{1,n}/(q_{1,n}-1)\rightarrow \infty$ as $q_{1,n}\rightarrow 1$. \color{black}(As shown in Figure \ref{fig:figure2}, the set $\mathcal{C}$ is unbounded.)\color{black}
	
	Following we prove the first equality, the second follows from it, as the union in the first equality is disjoint. To prove $"\supseteq"$, by the above argument, it suffices to prove   for  $(\mu, \alpha)\in \mathcal{U}'_2\subseteq \mathcal{B}'$, we have
	$$\Phi^{-1}((\mu, \alpha))=(q_0(\mu, \alpha), q_1(\mu, \alpha))\in \mathcal{U}_2.$$ It is true by Theorem \ref{t:main1}.  Furthermore, $(\mu, \alpha)\neq(0^\infty,1^\infty)$  means $\Phi^{-1}((\mu, \alpha))\notin \mathcal{C}$. Therefore  $\Phi^{-1}(\mathcal{U}'_2)\cap \mathcal{C}=\emptyset$.
	
	It remains to prove the inverse inclusion. If $(q_0, q_1)\in  \mathcal{U}_2\setminus  \mathcal{C}$, by applying  Theorem \ref{t:main1}  again,  we have,
	$$\Phi((q_0, q_1))=(\mu(q_0, q_1), \alpha(q_0, q_1) ).$$ Following we prove $(\mu(q_0, q_1), \alpha(q_0, q_1))\in\mathcal{U}_2'$. Clearly, $\mu(q_0, q_1)\neq 0^\infty$ or $\alpha(q_0, q_1)\neq 1^\infty$. 
	Note that   $\mu(q_0, q_1)$ and $\alpha(q_0, q_1)$ are the unique double-base expansion of the points $\ell_{q_0, q_1}$ and $r_{q_0, q_1}$, respectively.
	Using \eqref{e:unique-expanison}, it suffices to prove 	\begin{equation*}
		\left\{
		\begin{array}{clrr}
			\mu(q_0,q_1)\prec\sigma^m(\mu(q_0,q_1)) & \text{ whenever }\mu_m(q_0,q_1)=0, \\
			\sigma^m(\mu(q_0,q_1))\prec\alpha(q_0,q_1) & \text{ whenever } \mu_m(q_0,q_1)=1.\\
		\end{array}
		\right.
	\end{equation*}
	and 	\begin{equation*}
		\left\{
		\begin{array}{clrr}
			\mu(q_0,q_1)\prec\sigma^m((\alpha(q_0,q_1)) & \text{ whenever }\alpha_m(q_0,q_1)=0, \\
			\sigma^m(\alpha(q_0,q_1))\prec\alpha(q_0,q_1) & \text{ whenever } \alpha_m(q_0,q_1)=1.\\
		\end{array}
		\right.
	\end{equation*}
	We only prove the first group of inequalities. The second group  may be proved by a similar argument.
	\color{black}
	
	Suppose $\mu_m(q_0, q_1) = 0$. If there exists a largest index $k$ such that $\mu_k(q_0, q_1) = 1$ with $2 \leq k < m$, then
	$$
	\mu_k(q_0, q_1) \cdots \mu_m(q_0, q_1) = 10^{m-k}, \quad \sigma^k(\mu(q_0, q_1)) = 0^{m-k} \mu_{m+1}(q_0, q_1) \mu_{m+2}(q_0, q_1) \cdots.
	$$
	Thus, we have
	$$
	\mu(q_0, q_1) \prec \sigma^k(\mu(q_0, q_1)) \prec \cdots \prec \mu_{m+1}(q_0, q_1) \mu_{m+2}(q_0, q_1) \cdots = \sigma^m(\mu(q_0, q_1)).
	$$
	If no such $k$ exists, i.e., $\mu_1(q_0, q_1) \cdots \mu_m(q_0, q_1) = 0^m$, then
	$$
	\sigma^m(\mu(q_0, q_1)) = \mu_{m+1}(q_0, q_1) \mu_{m+2}(q_0, q_1) \cdots,
	\text{ and }
	\mu(q_0, q_1) \prec \sigma^m(\mu(q_0, q_1)).
	$$
	Therefore, we have shown that $\mu(q_0, q_1) \prec \sigma^m(\mu(q_0, q_1))$ whenever $\mu_m(q_0, q_1) = 0$.
	
	Next, suppose $\mu_m(q_0, q_1) = 1$. If there exists a largest index $k$ such that $\mu_k(q_0, q_1) = 0$ with $1 \leq k < m$, then
	$$
	\mu_k(q_0, q_1) \cdots \mu_m(q_0, q_1) = 01^{m-k}, \quad \sigma^k(\mu(q_0, q_1)) = 1^{m-k} \mu_{m+1}(q_0, q_1) \mu_{m+2}(q_0, q_1) \cdots.
	$$
	Thus, we get
	$$
	\sigma^m(\mu(q_0, q_1)) = \mu_{m+1}(q_0, q_1) \mu_{m+2}(q_0, q_1) \cdots \prec\sigma^k(\mu(q_0, q_1)) \prec \alpha(q_0, q_1).
	$$
	Therefore, $\sigma^m(\mu(q_0, q_1)) \prec \alpha(q_0, q_1)$ whenever $\mu_m(q_0, q_1) = 1$.
\end{proof}
\color{black}

\begin{lemma}\label{l:V_closed}
	$\mathcal{V}_2$ is closed. Moreover,
	$$\mathcal{V}_2=
	\Phi^{-1}(\mathcal{V}'_2\setminus\mathcal{C}')\cup \mathcal{C}
	\text{ and } \mathcal{V}_2\setminus\mathcal{C}=\Phi^{-1}(\mathcal{V}'_2\setminus\mathcal{C}').$$
\end{lemma}
\begin{proof}
	First we verify that the set $\mathcal{V}_2$ is closed by showing that its  complement  in $(1,\infty)^2$ is open. Consider $(q_0, q_1)\in (1,\infty)^2\setminus \mathcal{V}_2$. Then $(q_0, q_1)\in\{(q_0, q_1)\in (1,\infty)^2:q_0+q_1<q_0q_1\}$ or there exists some $j>0$ such that $$\sigma^j(\alpha(q_0, q_1))\prec\mu(q_0, q_1)\text{ or } \sigma^j(\mu(q_0, q_1))\succ\alpha(q_0, q_1).$$
	
	It is clear that the set $\{(q_0, q_1)\in (1,\infty)^2:q_0+q_1<q_0q_1\}$ is open.
	We only verify the case where \( \sigma^j(\alpha(q_0, q_1)) \prec \mu(q_0, q_1) \); the other case can be proved similarly. Let \( m > 1 \) be the smallest positive integer such that
	\begin{align}
		\alpha_{j+1}(q_0, q_1) \cdots \alpha_{j+m}(q_0, q_1) \prec \mu_1(q_0, q_1) \cdots \mu_m(q_0, q_1). \label{eq:first}.
	\end{align}
	
	Assume \( (q'_0, q'_1) \) is sufficiently close to \( (q_0, q_1) \). We have the following possible cases:
	\color{black}
	\subsection*{Case 1: \( \alpha(q'_0, q'_1) \prec \alpha(q_0, q_1) \) and \( \mu(q'_0, q'_1) \succ \mu(q_0, q_1) \)}
	
	By Lemma \ref{l:LRContinous}, \( \alpha(q'_0, q'_1) \) and \( \alpha(q_0, q_1) \) begin with the same word of length \( j + m \), and \( \mu(q'_0, q'_1) \) and \( \mu(q_0, q_1) \) begin with the same word of length \( m \). Therefore, \eqref{eq:first} still holds, and we conclude that \( (q'_0, q'_1) \notin \mathcal{V}_2 \).
	
	\subsection*{Case 2: \( \alpha(q'_0, q'_1) \prec \alpha(q_0, q_1) \) and \( \mu(q'_0, q'_1) \preccurlyeq \mu(q_0, q_1) \)}
	
	Applying Lemma \ref{l:LRContinous} again, we conclude that \( \lambda(q'_0, q'_1) \) and \( \lambda(q_0, q_1) \) begin with the same word of length \( m \).
	
	\begin{itemize}
		\item If $ \lambda(q_0, q_1) = \mu(q_0, q_1) $, then  $ \lambda(q'_0, q'_1) \preccurlyeq \mu(q'_0, q'_1) \preccurlyeq \mu(q_0, q_1) = \lambda(q_0, q_1)$, hence, \( \mu(q'_0, q'_1) \) and \( \mu(q_0, q_1) \) begin with the same word of length \( m \), and we conclude that \( (q'_0, q'_1) \notin \mathcal{V}_2 \).
		\item If \( \lambda(q_0, q_1) \) is co-finite, i.e., \( \lambda(q_0, q_1) = \lambda_1(q_0, q_1) \cdots \lambda_{\ell-1}(q_0, q_1) 01^\infty \) for some \( \ell \geq 2 \) and \( \mu(q_0, q_1) = (\lambda_1(q_0, q_1) \cdots \lambda_{\ell-1}(q_0, q_1) 1)^\infty \) by \eqref{eq:qlazy}, then by Lemma \ref{l:LRContinous}, we get $ \lambda(q'_0, q'_1), \mu(q'_0, q'_1) \text{ start with } \lambda_1(q_0, q_1) \cdots \lambda_{\ell-1}(q_0, q_1) 01^n0 \quad \text{for some large } n.$
		Now, using \eqref{eq:first}, we observe that \( \alpha_{j+1}(q_0, q_1) = \mu_1(q_0, q_1) = 0 \). Therefore, we have
		$$
		\sigma^\ell(\mu(q'_0, q'_1)) = 1^n 0 \cdots \succ \alpha_1(q_0, q_1) \cdots \alpha_j(q_0, q_1) 0 \cdots = \alpha(q_0, q_1) \succcurlyeq \alpha(q'_0, q'_1)  \text{ for }  n \geq j+1.
		$$
		Thus, \( (q'_0, q'_1) \notin \mathcal{V}_2 \).
	\end{itemize}
	
	The left two cases, \( \alpha(q'_0, q'_1) \succcurlyeq \alpha(q_0, q_1) \) and \( \mu(q'_0, q'_1) \succ \mu(q_0, q_1) \), as well as \( \alpha(q'_0, q'_1) \succcurlyeq \alpha(q_0, q_1) \) and \( \mu(q'_0, q'_1) \preccurlyeq \mu(q_0, q_1) \), can be verified by similar arguments.
	
	Finally, we prove the first equality, the second follows from it, as the union in the first equality is disjoint. The inclusion $"\supseteq"$ is obvious.		
	For the inverse inclusion, if $(q_0, q_1)\in \mathcal{V}_2\setminus\mathcal{C},$ then $\mu(q_0, q_1)\neq 0^\infty$ or $\alpha(q_0, q_1)\neq 1^\infty$. That is $\Phi((q_0, q_1))\in \mathcal{V}_2'\setminus\mathcal{C}'$ and then $(q_0, q_1)\in \Phi^{-1}(\mathcal{V}'_2\setminus\mathcal{C}')$.
\end{proof}

\begin{lemma}\label{l:Vdis2} Let $(q_0, q_1)\in \mathcal{V}_2$. If $$\mu(q_0, q_1)=(\mu_1\cdots\mu_m\alpha_1\cdots\alpha_n)^\infty\qtq{and}
	\alpha(q_0, q_1)=(\alpha_1\cdots\alpha_n\mu_1\cdots\mu_m)^\infty,$$then $(q_0, q_1)$ is an isolated point of $\mathcal{V}_2$. Hence $(q_0, q_1)\notin \overline{\mathcal{U}}_2$.
\end{lemma}
\begin{proof}
	The greedy double-base expansion of the point \( r_{q_0, q_1} \) is \( \alpha_1 \cdots \alpha_n \mu_1 \cdots \mu_m^+ 0^\infty \) and the lazy double-base expansion of the point \( \ell_{q_0, q_1} \) is \( \mu_1 \cdots \mu_m \alpha_1 \cdots \alpha_n^- 1^\infty \) by \eqref{eq:qlazy} and \eqref{eq:qgreedy2}. Thus, \( \alpha(q'_0, q'_1) \) begins with either \( \alpha_1 \cdots \alpha_n \mu_1 \cdots \mu_m^+ 0^{m+n} \) or \( (\alpha_1 \cdots \alpha_n \mu_1 \cdots \mu_m)^k \), and \( \mu(q'_0, q'_1) \) begins with either \( \mu_1 \cdots \mu_m \alpha_1 \cdots \alpha_n^- 1^{m+n} \) or \( (\mu_1 \cdots \mu_m \alpha_1 \cdots \alpha_n)^k \) by Lemma \ref{l:LRContinous}. However, \( \alpha(q'_0, q'_1) \) can't begin with \( \alpha_1 \cdots \alpha_n \mu_1 \cdots \mu_m^+  0^{m+n} \), nor can \( \mu(q'_0, q'_1) \) begin with \( \mu_1 \cdots \mu_m \alpha_1 \cdots \alpha_n^- 1^{m+n} \) since \( (q'_0, q'_1) \in \mathcal{V}_2 \). Thus, we conclude that \( \alpha(q'_0, q'_1) \) begins with \( (\alpha_1 \cdots \alpha_n \mu_1 \cdots \mu_m)^k \) and \( \mu(q'_0, q'_1) \) begins with \( (\mu_1 \cdots \mu_m \alpha_1 \cdots \alpha_n)^k \). 
	
	By an analogous argument from Lemma \ref{l: V'isolated}, we obtain \( \mu(q'_0, q'_1) = \mu(q_0, q_1) \) and \( \alpha(q'_0, q'_1) = \alpha(q_0, q_1) \). Hence, we conclude \( (q'_0, q'_1) = (q_0, q_1) \) by Theorem \ref{t:main1}.
	\color{black}
\end{proof}

\begin{lemma}\label{l:UU}
	$ \mathcal{U}_2\subsetneq\overline{\mathcal{U}_2}\subsetneq \mathcal{V}_2$. Moreover,
	$$\overline{\mathcal{U}_2}=
	\Phi^{-1}(\overline{\mathcal{U}'_2}\setminus\mathcal{C}')\cup \mathcal{C}
	\text{ and } \overline{\mathcal{U}_2}\setminus \mathcal{C}=\Phi^{-1}(\overline{\mathcal{U}'_2}\setminus\mathcal{C}').$$
\end{lemma}
\begin{proof}
	We know from Lemmas \ref{l:Ul2}-\ref{l:Vdis2} that  $\mathcal{U}_2 \subsetneq \overline{\mathcal{U}_2} \subsetneq \mathcal{V}_2.$
	Then we prove the first equality, the second follows from it, as the union in the first equality is disjoint. The $"\subseteq"$ is from Theorem \ref{t:main2} (iii), and  Lemmas \ref{l:V_closed} and \ref{l:Vdis2}. It remains to prove $\Phi^{-1}(\overline{\mathcal{U}'_2}\setminus\mathcal{C}')\subseteq \overline{\mathcal{U}_2}\setminus \mathcal{C}$.
	It follows  Lemma \ref{l:Ul2} that $ \Phi^{-1}(\mathcal{U}'_2)\subseteq \overline{\mathcal{U}_2}\setminus \mathcal{C}$.
	It remains to prove $ \Phi^{-1}(\overline{\mathcal{U}'_2}\setminus(\mathcal{C}'\cup \mathcal{U}'_2))\subseteq\overline{\mathcal{U}_2}\setminus\mathcal{C}$. For $(\mu, \alpha)\in \overline{\mathcal{U}'_2}\setminus(\mathcal{C}'\cup \mathcal{U}'_2)\subseteq \mathcal{B}'$,
	$\Phi^{-1}((\mu, \alpha))=(q_0(\mu, \alpha),q_1((\mu, \alpha)))$. By Lemma \ref{l:U'}, there exist $(\mu^n, \alpha^n)\in\mathcal{U}'_2$ such that $(\mu^n, \alpha^n)\rightarrow (\mu, \alpha)$ as $n$ goes to infinity. Clearly, we have
	$$\Phi^{-1}((\mu^n, \alpha^n))=(q_0(\mu^n, \alpha^n),q_1(\mu^n, \alpha^n))\in\mathcal{U}_2 $$  by Lemma  \ref{l:Ul2} and 	$(q_0(\mu^n, \alpha^n),q_1(\mu^n, \alpha^n))\rightarrow (q_0(\mu, \alpha), q_1(\mu, \alpha))$ since the map $\Phi^{-1}$ is continuous. That is $ \Phi^{-1}(\overline{\mathcal{U}'_2}\setminus(\mathcal{C}'\cup \mathcal{U}'_2))\subseteq\overline{\mathcal{U}_2}\setminus \mathcal{C}$.
\end{proof}

\begin{lemma}\label{l:U'uncountable}
	$\overline{\mathcal{U}'_2} \setminus\mathcal{U}'_2$ is uncountable. 
\end{lemma}
\begin{proof}
	Using  Lemmas \ref{l:Ul2} and \ref{l:UU} and the fact that $\Phi$ is bijective, we have
	\begin{equation}\label{e:Ul22}
		\begin{split}
			\overline{\mathcal{U}_2} \setminus\mathcal{U}_2&=\Phi^{-1}(\overline{\mathcal{U}'_2}\setminus\mathcal{C}')\setminus\Phi^{-1}( \mathcal{U}'_2)=\Phi^{-1}(\overline{\mathcal{U}'_2}\setminus(\mathcal{C}'\cup \mathcal{U}'_2)).\\
		\end{split}
	\end{equation}
	Moreover, $\overline{\mathcal{U}_2} \setminus\mathcal{U}_2$ is uncountable by Theorem \ref{t:hausdorff dimension} (see above), which together with \eqref{e:Ul22} leads to $\Phi^{-1}(\overline{\mathcal{U}'_2}\setminus(\mathcal{C}'\cup \mathcal{U}'_2))$ is uncountable. Therefore, $\overline{\mathcal{U}'_2}\setminus(\mathcal{C}'\cup \mathcal{U}'_2)$ and then $\overline{\mathcal{U}'_2} \setminus\mathcal{U}'_2$ are uncountable.
\end{proof}

\begin{lemma}\label{l:UlVldense}
	$\overline{\mathcal{U}_2} \setminus \mathcal{U}_2$ is dense in $\overline{\mathcal{U}_2}$; $\mathcal{V}_2 \setminus \overline{\mathcal{U}_2}$ is dense in $\mathcal{V}_2\setminus\mathcal{C}$.
\end{lemma}

\begin{proof}
	To show $\overline{\mathcal{U}_2} \setminus \mathcal{U}_2$ is dense in $\overline{\mathcal{U}_2}$, \color{black}
	it suffices to prove  $\mathcal{U}_2$ belongs to the closure of    
	$\overline{\mathcal{U}_2} \setminus \mathcal{U}_2$. \color{black} We recall from Lemma \ref{l:Ul2} that $\mathcal{U}_2=\Phi^{-1}(\mathcal{U}'_2)\cup \mathcal{C}.$ Assume $(\mu, \alpha)\in\mathcal{U}'_2$.  It follows from  Lemma \ref{l:Uldense} that  there exist $(\mu^k, \alpha^j)\in \overline{\mathcal{U}'_2} \setminus (\mathcal{U}'_2\cup\mathcal{C}')$ such that $(\mu^k, \alpha^j)\rightarrow (\mu, \alpha)\text{ when }k, j\rightarrow \infty$. Using \eqref{e:Ul22} and the inverse map $\Phi^{-1}$ is continuous, we have
	\begin{equation*}\label{e:Ul23}
		\begin{split}
			\Phi^{-1}((\mu^k, \alpha^j))=(q_0(\mu^k, \alpha^j),q_1(\mu^k, \alpha^j))\in \overline{\mathcal{U}_2} \setminus \mathcal{U}_2\rightarrow \Phi^{-1}(\mu, \alpha)\in \Phi^{-1}(\mathcal{U}'_2) \\
		\end{split}
	\end{equation*}
	as	$k, j\rightarrow \infty $.
	\color{black} Therefore, $ \Phi^{-1}(\mathcal{U}'_2)$ belongs to the closure of  $\overline{\mathcal{U}_2} \setminus \mathcal{U}_2$.	\color{black}
	
	Now we prove \color{black} $\mathcal{C}$ belongs to the closure of  
	$\overline{\mathcal{U}_2} \setminus \mathcal{U}_2$.	\color{black} Suppose  $(q_0, q_1)\in \mathcal{C}$  and $\epsilon$ is an arbitrary small number. We show that if $d_{\mathbb{R}^2}((q_0, q_1), (q'_0, q'_1))<\epsilon$ for some $(q'_0, q'_1)\in \mathcal{B}$, then  $(q'_0, q'_1)\in \mathcal{V}_2\cap\mathcal{B}$. It is true by the proof of Lemma \ref{l:V_closed}, there we have proved if $(q'_0, q'_1)\in \mathcal{B}$ but not in $\mathcal{V}_2$ and   $(q_0, q_1)$ is  close enough to  $(q'_0, q'_1)$, then $(q_0, q_1)\notin \mathcal{V}_2$. It contradicts to $(q_0, q_1)\in \mathcal{V}_2$.
	Moreover, we claim $(q'_0, q'_1)\notin \mathcal{V}_2\setminus\overline{\mathcal{U}_2}$. It is true since $\mathcal{V}_2\setminus\overline{\mathcal{U}_2}$ is a discrete set; see Lemma \ref{l:Vdis2}. Therefore, $(q'_0,q'_1)\in\overline{\mathcal{U}_2}\cap\mathcal{B}$.
	
	If $(q'_0,q'_1)\in\overline{\mathcal{U}_2}\setminus\mathcal{U}_2$, we obtain the result. If $(q'_0,q'_1)\in\Phi^{-1}(\mathcal{U}'_2)=\mathcal{U}_2\setminus \mathcal{C}$, then we may choose $(q''_0,q''_1)\in\overline{\mathcal{U}_2} \setminus \mathcal{U}_2$ to replace it, since \color{black}  $ \Phi^{-1}(\mathcal{U}'_2)$ belongs to the closure of  $\overline{\mathcal{U}_2} \setminus \mathcal{U}_2$ \color{black} by the above argument.
	
	To prove  $\mathcal{V}_2 \setminus \overline{\mathcal{U}_2}$ is dense in $\mathcal{V}_2\setminus\mathcal{C}=(\mathcal{V}_2 \setminus \overline{\mathcal{U}_2})\cup(\overline{\mathcal{U}_2}\setminus\mathcal{C})$, \color{black}it remains to prove  $\overline{\mathcal{U}_2}\setminus\mathcal{C}$ belongs to the closure of $\mathcal{V}_2 \setminus \overline{\mathcal{U}_2}$. \color{black}
	It follows from Lemmas \ref{l:Ul2} and \ref{l:UU} that
	\begin{equation}\label{e:relation1}
		\mathcal{U}_2\setminus\mathcal{C}=\Phi^{-1}(\mathcal{U}'_2)\text{ and }\overline{\mathcal{U}_2}\setminus\mathcal{C}=\Phi^{-1}(\overline{\mathcal{U}'_2}\setminus\mathcal{C}').
	\end{equation}
	We obtain from Lemma \ref{l:U'} that 
	\color{black} $\overline{\mathcal{U}'_2}\setminus\mathcal{C}'$ belongs to the closure of $\mathcal{U}'_2$. 
	This together with the continuity  of the map  $\Phi^{-1}$ and \eqref{e:relation1} implies that $\overline{\mathcal{U}_2}\setminus\mathcal{C}$ belongs to the closure of $\mathcal{U}_2\setminus\mathcal{C}$.
	Hence  it remains to prove ${\mathcal{U}_2}\setminus\mathcal{C}$ belongs to the closure of $\mathcal{V}_2 \setminus \overline{\mathcal{U}_2}$.  \color{black}
	
	We recall from Lemma \ref{l:dense22} that 	\color{black}  $\mathcal{U}'_2$ belongs to the closure of $\mathcal{V}'_2 \setminus \overline{\mathcal{U}'_2}$\color{black}. Similarly, applying the continuity  of the map  $\Phi^{-1}$  again, we have \color{black} $\Phi^{-1}(\mathcal{U}'_2)={\mathcal{U}_2}\setminus\mathcal{C}$ belongs to the closure of $\Phi^{-1}(\mathcal{V}'_2\setminus\overline{\mathcal{U}'_2})=\mathcal{V}_2 \setminus \overline{\mathcal{U}_2}$.  \color{black}
\end{proof}

Define
\begin{equation*}
	\begin{split}
		&\mathcal{A}_1:=\left([1,2]\times [1,2]\right)\cap \overline{\mathcal{U}_2},\\
		&\mathcal{A}_n:=\left([n,n+1]\times [1,\frac{n}{n-1}]\right)\cap \overline{\mathcal{U}_2},\\
		&\widetilde{\mathcal{A}_n}:=\left([1,\frac{n}{n-1}]\times [n,n+1]\right)\cap \overline{\mathcal{U}_2}
	\end{split}
\end{equation*} 
for all $ n=2,\cdots$.
\begin{lemma}\label{l:partition}
	$\overline{\mathcal{U}_2}=\mathcal{A}_1\cup\bigcup_{n=2}^\infty(\mathcal{A}_n\cup \widetilde{\mathcal{A}_n})$.
\end{lemma}
\begin{proof}
	It suffices to prove $"\subseteq"$. For each $(q_0, q_1)\in \overline{\mathcal{U}_2}$, there exists an $n\in \mathbb{N}$ such that 
	$q_0\in [n, n+1]$. Furthermore, it follows from $q_0+q_1\ge q_0q_1$ that $q_1\le q_0/(q_0-1)$. Then we obtain that $q_1\in (1,\infty)$, if $q_0\in (1,2]$; $q_1\in[1, n/(n-1)]$, if $q_0\in[n, n+1]$ for all $n=2,\cdots$. Hence, $(q_0, q_1)\in \bigcup_{n=2}^\infty\mathcal{A}_n$ if $n\geq 2$. 
	
	In the following we show that, for $n=1$
	\begin{equation}\label{e:p1}
		(q_0, q_1)\in \mathcal{A}_1\cup\bigcup_{n=2}^\infty\widetilde{\mathcal{A}_n}.
	\end{equation}
	We divide $[1,\infty)$ as $[1,\infty)=\bigcup_{n=1}^\infty [n, n+1]$. By a similar argument as above, we obtain that $q_0\in [1,n/(n-1)]$ if $q_1\in[n,n+1]$ for all $n>1.$ Therefore \eqref{e:p1} holds.	
\end{proof}

\begin{proposition}\label{p:cantor}
	$\overline{\mathcal{U}_2}$ is a countable union of Cantor sets. 
\end{proposition}
\begin{proof}
	\color{black}  We begin by noting that $\overline{\mathcal{U}_2}$ has no isolated points, which follows from Lemma \ref{l:UlVldense}. Next, we show that $\overline{\mathcal{U}_2}$ has no interior points. \color{black} Assume on the contrary that $\overline{\mathcal{U}_2}$ has an interior point $ (q_0, q_1) $.  Then $ \overline{\mathcal{U}_2} $ contains an open disk of radius $ r $, centered at $(q_0, q_1) $. Then the concentric open disk of radius $r/2$ belongs to $\overline{\mathcal{U}_2} \setminus \mathcal{C} $, and therefore $ (q_0, q_1) $ is also an interior point of $ \overline{\mathcal{U}_2} \setminus \mathcal{C}$, which is a contradiction by Lemma \ref{l:UlVldense}. \color{black} 
	Hence, we  obtain that $\mathcal{A}_n$ for $n=1,\cdots$ and $\widetilde{\mathcal{A}_n}$ for $n=2,\cdots$ are Cantor sets. 
\end{proof}
\color{black}
\begin{lemma}\label{l:thin}
	$\mathcal{U}_2$, $\overline{\mathcal{U}_2}$  and $\mathcal{V}_2$ are  nowhere dense.
\end{lemma}
\begin{proof}
	\color{black}  First, $\overline{\mathcal{U}_2}$ has no interior points by Proposition \ref{p:cantor}, which directly implies that $ \mathcal{U}_2 $ and $ \overline{\mathcal{U}_2} $  are nowhere dense, as their closure both have no interior points. Similarly, by using the above facts $\overline{\mathcal{U}_2}$ has no interior points and Lemma \ref{l:Vdis2}, we establish that $ \mathcal{V}_2 $ is also a nowhere dense . \color{black}
\end{proof}
\color{black}

\begin{proof}[Proof of (vii) of Theorem \ref{t:main2} ]
	It follows from Lemma \ref{l:U'uncountable}.
\end{proof}

\begin{proof}[Proof of Theorem \ref{t:main3} ]
	It follows from Lemmas \ref{l:Ul2}, \ref{l:V_closed} and  \ref{l:UU}.
\end{proof}

\begin{proof}[Proof of (i)-(iv) of Theorem \ref{t:main4}]
	(i)-(iii) follows from Lemmas \ref{l:Ul2}, \ref{l:V_closed},
	\ref{l:UU} and \ref{l:thin}, and Proposition \ref{p:cantor}; (iv) is from Lemmas \ref{l:Vdis2}, \ref{l:UlVldense}, ${\mathcal{V}_2} \setminus\overline{\mathcal{U}_2}=\Phi^{-1}(\mathcal{V}'_2\setminus\overline{\mathcal{U}'_2})$ and ${\mathcal{V}'_2}\setminus\overline{\mathcal{U}'_2}$ is countable.
\end{proof}	

\section{Open problems}
We end this paper by formulating some open problems:
\begin{enumerate}[($i$)]
	\item Does the sets $\mathcal{U}_2$, $\overline{\mathcal{U}_2}$ and  $\mathcal{V}_2$  have null Lebesgue measure?
	\item  For the set $\overline{\mathcal{U}_2} \setminus \mathcal{U}_2$,  we showed that $\dim_H(\overline{\mathcal{U}_2} \setminus \mathcal{U}_2)\ge 1$. \color{black} What is the exact value of the dimension?\color{black}
	
\end{enumerate}

\section{Acknowledgments}
The authors thank the referees for many helpful comments and suggestions concerning the presentation of their results, and also express our gratitude to Vilmos Komornik for providing a more concise and comprehensible proof for Proposition \ref{p1}, which greatly improved the readability of the article.
\noindent

\end{document}